\documentclass[10pt,a4paper]{amsart}
\usepackage[latin1]{inputenc}
\usepackage[T1]{fontenc}
\usepackage{amsmath,amscd}
\usepackage{graphicx}
\usepackage[all]{xy}
\usepackage{amsthm}
\usepackage{amssymb,url}
\usepackage[charter]{mathdesign}
\usepackage{bbm}
\usepackage{ae}

\newtheorem{sats}{Theorem}[section]

\newtheorem{sats*}{Theorem}

\newtheorem{lem}[sats]{Lemma}
\newtheorem{cor}[sats]{Corollary}
\newtheorem{prop}[sats]{Proposition}

\newcommand{\R}{\mathbbm{R}}
\newcommand{\C}{\mathbbm{C}}

\newcommand{\Z}{\mathbbm{Z}}

\newcommand{\N}{\mathbbm{N}}

\newcommand{\ellL}{\mathcal{L}}

\newcommand{\ch}{ \mathrm{ch}}
\newcommand{\id}{ \mathrm{id}}
\newcommand{\rd}{\mathrm{d}}

\newcommand{\I}{\mathcal{I}}

\newcommand{\Ko}{\mathcal{K}}

\newcommand{\Bo}{\mathcal{B}}
\newcommand{\Co}{\mathcal{C}}
\newcommand{\He}{\mathcal{H}}

\newcommand{\Fg}{\mathcal{F}}

\newcommand{\Eg}{\mathcal{E}}

\newcommand{\ind}{\mathrm{i} \mathrm{n} \mathrm{d} \,}

\renewcommand{\epsilon}{\varepsilon}
\renewcommand{\phi}{\varphi}

\newcommand{\im}{\mathrm{i} \mathrm{m} \,}

\newcommand{\e}{\mathrm{e}}
\newcommand{\tra}{\mathrm{t}\mathrm{r}}

\title[Operator-valued twisted index pairing]{Operator-valued pseudo-differential operators and the twisted index pairing}
\author{Magnus Goffeng}

\begin{document}

\address{Mathematical Sciences, Chalmers University of Technology and the University of Gothenburg \newline
\indent 412 96 G\"oteborg \newline
\indent Sweden}

\address{Institut f\"ur Analysis, Leibniz Universit\"at Hannover\newline
\indent 301 67 Hannover\newline
\indent Germany\newline}

\email{goffeng@math.uni-hannover.de}

\subjclass{Primary: 19K56, Secondary: 19L50, 46L80}

\keywords{Projective pseudo-differential operators, twisted $K$-theory, twisted index pairing, continuous trace algebras, $T$-duality, magnetic translations.}

\maketitle

\bibliographystyle{amsplain}

\begin{abstract}
The notion of pseudo-differential operators with coefficients in a continuous trace algebra over a manifold are introduced and their index theory is studied. The algebra of principal symbols in this calculus provides an abstract Poincar\'e dual to the continuous trace algebra. Index formulas for pseudo-differential operators twisted by a bundle on the opposite continuous trace algebra are obtained in terms of an Atiyah-Singer type index formula describing the twisted index pairing as index theory for elliptic operators.  
\end{abstract}

\Large
\section*{Introduction}
\normalsize

Twisted $K$-theory have many applications in theoretical physics, differential geometry and operator algebras. In theoretical physics, a $D$-brane is a twisted geometric $K$-homology cycle on spacetime and the charge of a $D$-brane is the Poincar\'e dual class in twisted $K$-theory, see more in \cite{bmrs}, \cite{carwang} and \cite{rosentexas}. In differential geometry, twisted $K$-theory can be described as the $K$-theory of modules of a differential gerbe, see more in \cite{bocamava} and \cite{murray}. The approach that will play a more important role in this paper is the description of twisted $K$-theory in terms of operator algebras, where twisted $K$-theory comes from the $K$-theory of a continuous trace algebra. A continuous trace algebra is a locally trivial bundle of compact operators over a topological space, and a stable continuous trace algebra is always the algebra of sections of an infinite-dimensional Azumaya bundle, so over every connected component it comes from a principal $PU(\He)$-bundle for some Hilbert space $\He$.

The objects of study in this paper are pseudo-differential operators with coefficients in a continuous trace algebra over a manifold and in particular the associated $K$-theoretic invariants and indices. The $K$-theory of the principal symbol algebra of this pseudo-differential calculus is a twisted $K$-theory group. This class of pseudo-differential operators gives a new interpretation of the twisted index pairing which allows for explicit calculations in many examples. As in the non-twisted case, the twisted index pairing is a bilinear pairing $K^*(X,\omega)\times K_*(X,\omega)\to \Z$. The twisted index pairing has previously been studied in \cite{carwang} using geometric twisted $K$-homology. A cohomological formula for the index pairing was obtained in \cite{carwang} that coincides with the Atiyah-Singer index mapping $K^*(T^*X)\to \Z$ applied to the two classes' cup product after using Poincar\'e duality $K_*(X,\omega)\cong K^*(T^*X,-\pi^*\omega)$, where $\pi:T^*X\to X$ denotes the projection.

Elements of the pseudo-differential calculus constructed in this paper are locally operator-valued pseudo-differential operators with symbols globally transforming as elements of the continuous trace algebra. The problem in constructing this calculus arise both locally and globally. 

Locally, there is an analytic problem with pseudo-differential operators acting on an infinite-dimensional fiber. To solve this problem we use the operator-valued pseudo-differential calculus of \cite{grigone} and \cite{grigtwo} which gives a good symbol calculus for zero order operators that admits Chern characters and index theory in the same manner as on a finite-dimensional fiber. The underlying analytic idea of this construction of operator-valued pseudo-differential operators is to consider only pseudo-differential operators that on the fiber behave like a pseudo-differential operator on a manifold. 

Globally there are two problems, how to represent the continuous trace algebra and how to define smoothness of symbols. To construct an operator from a symbol one must represent the symbol algebra. One can consider a very large Hilbert space, the Hilbert space of Hilbert-Schmidt sections, which introduces a huge degeneracy of these operators. Therefore, restricting to  subspaces coming from finitely generated projective modules of the opposite continuous trace algebra is necessary for obtaining Fredholm properties. Concerning smoothness, most examples of Hilbert space bundles and projective Hilbert space bundles that arise in practice do unfortunately have the compact-open topology on the structure group and not the technically simpler norm topology. Much of the general theory is restricted to the norm smooth bundles. Some of these problems can be adressed in examples by doing calculations over trivializing covers. \\

The paper is organized as follows; in the first section we recall the operator-valued pseudo-differential operators of \cite{grigone} and \cite{grigtwo}. The index theory for operator-valued pseudo-differential operators is needed in the third section to describe the index pairing in twisted $K$-theory. As an application we calculate indices of non-local elliptic pseudo-differential operators on a non-compact manifold equipped with a free cocompact action of a discrete group. 

We will in the second section, after collecting some known results for continuous trace algebras, introduce the projective symbol calculus. As examples of how to construct projective symbols we consider the two examples of $T$-duals of principal $S^1$-bundles $Z\to X$ and fiberwise magnetic translations on fiber bundles $Z\to X$ that are equivariant with respect to a discrete group, see subsection \ref{tdualbundle} respectively subsection \ref{twistedgroupbundle}.  We also calculate the $K$-theory of the algebra of projective principal symbols which turns out to be the twisted $K$-theory of the cotangent bundle, see Theorem \ref{gendd}. Using the ideas of \cite{grigtwo} and the generalized Connes-Hochschild-Kostant-Rosenberg theorem we can construct the Chern character of an elliptic projective pseudo-differential operator as an element of twisted de Rham cohomology. 

In the third and final section we will construct operators from projective symbols. This construction defines an isomorphism from the $K$-theory of the principal symbol algebra to the Kasparov group $KK(C(X,\Ko(P^{op})),\C)$, see Theorem \ref{chooseanop}, i.e. the principal symbol algebra of our calculus does in a sense provide an abstract Poincar\'e dual to $C(X,\Ko(P^{op}))$. The index pairing of an elliptic projective symbol with an element of $K_0(C(X,\Ko(P^{op})))$ is given by twisting the symbol with a smooth projection. The index theory from the non-twisted setting gives us a formula in twisted de Rham cohomology for the index of this operator, see Theorem \ref{indexpairing}. We calculate this index pairing for $T$-duals in Theorem \ref{tdualsymbols}, where the twisted index pairing is of a very simple form since it can be described as usual pseudo-differential operators on a manifold using the Thom-Connes isomorphism. We also study the twisted index pairing of elliptic projective operators equivariant under fiberwise magnetic translations with fiberwise elliptic operators that are equivariant under the fiberwise opposite magnetic translations. This produces index formulas in twisted de Rham cohomology for certain types of differential operators, see Theorem \ref{projectivetwiseddiscrete}.

\Large
\section{Operator-valued pseudo-differential operators}
\label{sectionone}
\normalsize

We will start by presenting an analytic framework for pseudo-differential operators acting on an infinite-dimensional fiber. The framework we choose is the operator-valued symbol calculus from \cite{grigone} and \cite{grigtwo}. There are also other, less restrictive, symbol calculuses available for infinite-dimensional fibers such as \cite{schulze} or \cite{luke}, but we choose that of \cite{grigone} and \cite{grigtwo} for its simplicity in calculating Chern characters. To avoid densely defined operators and domain issues we follow the convention of \cite{grigone} and \cite{grigtwo} and only consider operators of non-positive order. 

The unitary group on an infinite-dimensional Hilbert space is contractible in its compact-open topology by Kuiper's theorem, see \cite{kuiper}, so any vector bundle with fibers being infinite-dimensional Hilbert spaces is topologically trivializable.  The symbol calculus we consider in this section can in practice often be constructed also for Hilbert space bundles whose structure group is the group of unitaries in the compact-open topology but in later sections for projective bundles deeper problems occur whenever outside the norm topology. In general, it is only the $C^*$-algebra generated by the symbols that makes sense. In this section we will work in the simplest setting of trivial Hilbert space bundles, which will provide a straight-forward setting for norm-smooth projective Hilbert space bundles.\\

Let $\He$ denote a separable Hilbert space and for $p\geq 1$ we let $\ellL^p(\He)$ denote the Banach space of Schatten class operators of order $p$ on $\He$. When speaking of smooth mappings we will, unless otherwise stated, mean smooth in the appropriate norm sense. We recall the notion of an operator-valued symbol, Definition $2.1$ from \cite{grigtwo}. For $p>0$, a large integer $N$ and $m\leq 0$ we define the linear space $\Sigma^{m,p}_c(\R^n,\He)$ as the set of smooth functions $a:\R^{n}\times \R^n\to \Bo(\He)$ with compact support in the first $n$ variables that for $|\alpha|,|\beta|\leq N$ satisfy the estimates
\begin{align}
\label{firstest}
\|\partial^\alpha_\xi \partial^\beta_x a(x,\xi)&\|_{\Bo(\He)}\;\lesssim (1+|\xi|)^{m-|\alpha|},\\
\label{secondest}
\|\partial^\alpha_\xi \partial^\beta_x a(x,\xi)&\|_{\ellL^1(\He)}\lesssim (1+|\xi|)^{m+p-|\alpha|} \quad\mbox{for}\quad |\alpha|>m+p,\\
\label{thirdest}
\|\partial^\alpha_\xi \partial^\beta_x a(x,\xi)&\|_{\ellL^{p/(|\alpha|-m)}(\He)}<\infty \;\;\qquad\qquad\mbox{for}\quad m<|\alpha|\leq m+p.
\end{align}
The choice of $N$ will not play an important role as long as it is sufficiently large.

The conditions \eqref{firstest}-\eqref{thirdest} are motivated by the following situation: suppose that $a_0\in S^{m}(\R^{n+p})$, the H\"ormander class of symbols of order $m\leq 0$ on $\R^{n+p}$, has compact support in its first $n+p$ coordinates and set $a:(x,\xi)\mapsto a_0(x,y,\xi,\partial_y)\in \Bo(L^2(\R^p))$. The operator-valued symbol $a$ is an element of $\Sigma^{m,p+\epsilon}_c(\R^n,L^2(\R^p))$ for any $\epsilon>0$.  This fact is proved by observing that for $\alpha,\beta\in \N^n$ the operator-valued function $\partial^\alpha_\xi \partial^\beta_xa(x,\xi)$ takes values in the Schatten class operators of order $(p+\epsilon)/(|\alpha|-m)$ for any $\epsilon>0$ since it is defined from the pseudo-differential operator $\partial^\alpha_\xi \partial^\beta_xa(x,y,\xi,\eta)$ which is of order $m-|\alpha|$. The estimates \eqref{secondest} and \eqref{thirdest} follows from this observation, compare to Theorem $1.8$ of \cite{luke}. This example illustrates that the Schatten condition can be replaced by any suitable operator ideal $\I\subseteq\Bo(\He)$, such as any symmetric normed ideal. For instance, when $a$ is constructed as before from an $a_0\in S^m(\R^{n+p})$ the conditions \eqref{secondest} and \eqref{thirdest} holds when replacing the Schatten norms by Dixmier norms and $a\in \Sigma^{m,p+}_c(\R^n,L^2(\R^p))$. 

An operator-valued symbol $a\in \Sigma^{m,p}_c(\R^n,\He)$ defines a linear operator $a(x,\partial)$ which a priori is defined on elements $f\in \mathcal{S}(\R^n,\He)$ as
\[a(x,\partial)f(x):=\frac{1}{(2\pi)^n}\int _{\R^n} a(x,\xi)\exp(ix\cdot \xi)\hat{f}(\xi)\rd \xi,\]
where $\hat{f}$ denotes the Fourier transform of $f$.  The operator $a(x,\partial)$ is called an operator-valued pseudo-differential operators of order $(m,p)$. An operator-valued pseudo-differential operator extends to a bounded operator on $L^2(\R^n,\He)$ which is trace class if $m$ is small enough, see Proposition $3.2$ of \cite{grigone}. We define 
\begin{equation}
\label{prodform}
a\circ_N b(x,\xi):=\sum _{[\alpha|<N}\frac{1}{\alpha!} \left[(\partial_\xi^\alpha a)(x,\xi)\right]\cdot\left[(\partial_x^\alpha a)(x,\xi)\right].
\end{equation}
for $a,b\in \Sigma^{m,p}_c(\R^n,\He)$. Then $a(x,\partial)b(x,\partial)-(a\circ _N b)(x,\partial)\in \ellL^1(L^2(\R^n,\He))$. Therefore the  operator-valued pseudo-differential operators of order $(m,p)$ together with the trace class operators form an algebra which we denote by $\Psi^{m,p}_c(\R^n,\He)$. Furthermore, if $\chi,\chi'\in C^\infty_c(\R^n)$ has disjoint support then $\chi Q\chi'\in \ellL^1(L^2(\R^n,\He))$ for any $Q\in \Psi^{m,p}_c(\R^n,\He)$ and any $m,p$. For proofs of these statements we refer to \cite{grigone}. 

For operator-valued symbols, it is often more natural to define homogeneous symbols with respect to an $\R_+$-action on the fiber $\He$. Suppose that $\kappa:\R_+\to GL(\He)$ is a strongly continuous $\R_+$ action on $\He$. A symbol $a\in\Sigma^{m,p}_c(\R^n,\He)$ is called homogeneous of order $m$ with respect to $\kappa$ if $a(x,\lambda\xi)=\lambda^m\kappa(\lambda) a(x,\xi)\kappa(\lambda^{-1})$ for large $\xi$ and $\lambda\geq 1$. This definition can be found in \cite{schulze}. The $\R_+$-action on the fiber will not play an important role in this paper so we will only refer to homogeneous symbols of order $m$.

We adopt the convention from \cite{grigone} and \cite{grigtwo} and replace the infinitely smoothing operators by trace class operators. This convention simplifies many constructions. For the purpose of index theory, trace class operators does not affect anything and the same calculations are possible, so this coarse scale of smoothness suffices. 

Assume that $X$ is a smooth manifold without boundary. We will use the notation $\pi$ for the projection of the cotangent bundle $T^*X\to X$ and also for its restriction to the cosphere bundle $S^*X\to X$. Following the standard procedure, an operator-valued pseudo-differential operator of order $(m,p)$ is an operator $Q\in \Bo(L^2(X,\He))$ such that if the open subset $U\subseteq X$ lies in a coordinate chart and $\chi,\chi'\in C^\infty_c(U)$ then $\chi Q\chi'$ is, up to an element of $\ellL^1(L^2(U,\He))$, an operator-valued pseudo-differential operator of order $(m,p)$ in the local coordinate chart on $U$. By Theorem $2.8$ of \cite{grigtwo} the definition of an operator-valued pseudo-differential operator is independent of coordinate charts. Unless otherwise stated we will assume that our manifolds are compact.

We will denote the space of operator-valued pseudo-differential operators of order $(m,p)$ on $X$ by $\Psi^{m,p}(X,\Ko)$. The motivation for this notation is that the operator-valued pseudo-differential operator has coefficients in the multiplier algebra of the trivial bundle over $X$ with fiber $\Ko$. One could equally well consider a norm smooth Hilbert space bundle $\mathbbm{H}\to X$ and pseudo-differential operators with coefficients in the multipliers of the bundle of fiber-wise compact endomorphisms $\Ko(\mathbbm{H})\to X$. The problem is that the Hilbert space bundles arising in examples often are not norm smooth, let alone norm continuous, but rather strongly continuous. The same problem causes difficulties for the pseudo-differential operators with symbols in an Azumaya bundle that will appear later on in the paper. There are three ways out of this problem; either one finds an isomorphism to a norm-smooth bundle, which for a Hilbert space bundle can be taken to be trivial, or one loosens the smoothness restriction on the symbols and work with symbols from a $C^*$-algebra or one could choose local trivializations that glue together via transition functions that are smooth only in the strong sense. 

The space $\Psi^{0,p}(X,\Ko)$ forms a $*$-algebra filtered by $\Psi^{m,p}(X,\Ko)$, see Theorem $3.6$ of \cite{grigone} for a proof of this statement. If $U\subseteq X$ is an open subset we will let $\Psi^{m,p}_c(U,\Ko)\subseteq \Psi^{m,p}(X,\Ko)$ denote the subalgebra of operators that are proper and compactly supported in $U$, i.e. the operators with Schwartz kernels whose supports are compact in $U\times U$. The algebra $\Psi^{m,p}(X,\Ko)$ is spanned by the smoothing operators and all $\Psi^{m,p}_c(U,\Ko)$ for open subsets $U$ inside a coordinate neighborhood, since the operator-valued pseudo-differential operators are pseudo-local.

Let us introduce some notations for the symbol algebras that will appear throughout the paper. If $U$ is an open subset of a manifold we will let $C^{N,\infty}(U,\ellL^p(\He))$ denote $C^\infty(U,\Bo(\He))\cap C^N(U,\ellL^p(\He))$ and we denote the space of compactly supported functions of this space by $C_c^{N,\infty}(U,\ellL^p(\He))$. We will let $\Sigma^{m,p}(X,\Ko)$ denote the subspace of $C^\infty(T^*X,\Bo(\He))$ that satisfies \eqref{firstest}--\eqref{thirdest} in local coordinates. If $U$ is open in $X$ we will let $\Sigma_c^{m,p}(U,\Ko)$ denote the subspace of functions in $\Sigma^{m,p}(X,\Ko)$ that are supported over the base $U$, i.e. $a\in \Sigma^{m,p}_c(U,\Ko)$ if $\{x: a|_{\pi^{-1}(x)}\neq 0\}$ is a precompact subset of $U$. Finally, we will set $\Sigma^{m,p}_{cc}(X,\Ko):=\Sigma^{m,p}(X,\Ko)\cap C_c(T^*X,\Bo(\He))$. Let $\Psi^{m,p}_{cc}(X,\Ko)\subseteq \Psi^{m,p}(X,\Ko)$ denote the algebra of operator-valued pseudo-differential operators generated from symbols in $\Sigma^{m,p}_{cc}(X,\Ko)$.

Any operator $Q\in \Psi^{0,p}(X,\Ko)$ is pseudo-local. Thus there are finite covers $(U_\alpha)$ of $X$ and trivializations $f_\alpha:U_\alpha\to \R^n$ such that any $Q$ can be written as 
\begin{equation}
\label{locdecomp}
Q=\sum_\alpha  u_{\alpha}\left[a_\alpha(x,\partial)\right]u_{\alpha}^{-1} +S,
\end{equation}
for a collection $a_\alpha\in \Sigma^{0,p}_c(\R^n,\He)$ and an $S\in \ellL^1(L^2(X,\He))$. Here $u_\alpha:L^2(\R^n)\to L^2(U_\alpha)$ denotes the isomorphism constructed from $f_\alpha$. We will use the notation $\mathfrak{U}$ for the triple consisting of $(U_\alpha)$, $(f_\alpha)$ and a choice of a subordinate partition of unity $(\chi_\alpha)$. From the decomposition \eqref{locdecomp} we define the function $\sigma_\mathfrak{U}(Q,S)\in \Sigma^{m,p}(X,\Ko)$ by $\sigma_\mathfrak{U}(Q,S):=\sum_\alpha\chi_\alpha  f^*_\alpha (a_\alpha)$. We call this function the full symbol of $Q$ with respect to $S$ and $\mathfrak{U}$. Observe that $\sigma_\mathfrak{U}(Q,S)$ only depends linearly on the pair $(Q,S)$. 

We can decompose the full symbol by defining $\sigma^0_\mathfrak{U}(Q,S)(x,\xi):=\sigma_\mathfrak{U}(Q,S)(x,0)$ and $\sigma^1_\mathfrak{U}(Q,S)(x,\xi):=\sigma_\mathfrak{U}(Q,S)(x,\xi)-\sigma_\mathfrak{U}(Q,S)(x,0)$. Thus the full symbol $\sigma_\mathfrak{U}(Q,S)$ takes the form
\begin{equation}
\label{decompconstschatt}
\sigma_\mathfrak{U}(Q,S)=\sigma^0_\mathfrak{U}(Q,S)+\sigma^1_\mathfrak{U}(Q,S).
\end{equation}
Let us emphasize the obvious fact that the symbol $\sigma^0_\mathfrak{U}(Q,S)$ is independent of $\xi$. We can in each $T^*U_\alpha$ decompose the dual coordinates $\xi=(\xi_1,\xi')$ and using this decomposition we can rewrite 
\begin{equation}
\label{intedecomp}
\chi_\alpha\sigma^1_\mathfrak{U}(Q,S)(x,\xi)=\int_0^{\xi_1} \partial_{\xi_1}(\chi_\alpha \sigma_\mathfrak{U}(Q,S))(x,t,\xi')\rd t.
\end{equation}
Using this integral formula, \eqref{thirdest} implies that $\sigma^1_\mathfrak{U}(Q,S)\in C^{N,\infty}(T^*X,\ellL^p(\He))$ since the Schatten class operators form a Banach space. It is trivial that $\sigma^1_\mathfrak{U}(Q,S)$ also satisfies \eqref{firstest}-\eqref{thirdest}. A consequence of this decomposition is the following proposition:

\begin{prop}
\label{compellp}
We have an inclusion $\Sigma^{m,p}_{cc}(X,\Ko)\subseteq C^{N,\infty}_c(T^*X,\ellL^p)$.
\end{prop}

An interesting setting of operator-valued symbols is on a fiber bundle $\pi_Z:Z\to X$ with smooth compact fiber $M$ and $X$ is a compact smooth manifold. Any $0$-order pseudo-differential operator on $Z$ defines an operator-valued pseudo-differential operator on the Hilbert space bundle $L^2(Z|X)\to X$ associated with the fiber bundle by taking the fiber of $L^2(Z|X)$ over $x\in X$ as the Hilbert space $L^2(\pi_Z^{-1}(x))$. Here we see the problem addressed above since the gluing cocycles of this group only are continuous when the group of unitaries on $L^2(M)$ is equipped with the compact-open topology. Observe that a symbol $a_0=a_0(x,y,\xi,\eta)$ of order $m$ on $Z$ defines an operator-valued function  $a(x,\xi)=a_0(x,y,\xi,\partial_y)$ on $T^*X$, this construction will of course depend on the choice of coordinate system on the fibers of $Z$. The function $a$ will in a fixed coordinate system satisfy  the estimates \eqref{firstest}--\eqref{thirdest} for any $p>\dim(M)$. \\

We will in this paper mainly focus on classical pseudo-differential operators. In the context of operator-valued pseudo-differential operators a classical pseudo-differential operator is an element $Q\in \Psi^{m,p}(X,\Ko)$ such that there exist $Q_j\in \Psi^{m-j,p}(X,\Ko)$ which are homogeneous of order $m-j$ and $Q\sim \sum Q_j$. In our convention of smoothing operators, the asymptotic sum can be taken as a finite sum. The classical operator-valued pseudo-differential operators of order $(m,p)$ again form an algebra which we denote by $\Psi^{m,p}_{cl}(X,\Ko)$. 

The principal symbol mapping $Q\mapsto Q_0$, for $Q\sim \sum Q_j\in\Psi^{0,p}_{cl}(X,\Ko)$, does not depend on the choice of coordinates and defines a mapping $\sigma:\Psi^{0,p}_{cl}(X,\Ko)\to C^{\infty}(S^*X,\Bo(\He))$. The mapping $\sigma$ fits into a short exact sequence:
\begin{equation}
\label{psix}
0\to \tilde{\Psi}^{-1,p}_{cl}(X,\Ko)\to\Psi^{0,p}_{cl}(X,\Ko)\xrightarrow{\sigma}\Sigma^{p}(X,\Ko)\to 0,
\end{equation}
where we define $\Sigma^{p}(X,\Ko):=\im (\sigma:\Psi^{0,p}_{cl}(X,\Ko)\to C^{\infty}(S^*X,\Bo(\He)))$ and the ideal $\tilde{\Psi}^{-1,p}_{cl}(X,\Ko)$ is defined as $\ker\sigma$. Because of the technical assumptions we have put on the symbols, a compactly supported symbol is not the symbol of a smoothing operator or even an operator of lower order, although clearly defining a compact operator. This fact combined with the definition of homogeneous symbols gives the space $\ker \sigma$ a slightly complicated structure:

\begin{prop}
\label{homsymbprop}
The ideal $\tilde{\Psi}^{-1,p}_{cl}(X,\Ko)$ is the subalgebra of $\Ko(L^2(X,\He))$ linearly spanned by $\Psi^{-1,p}(X,\Ko)$ and $\Psi^{0,p}_{cc}(X,\Ko)$.
\end{prop}

\begin{proof}
If the operator $Q\in \Psi^{0,p}_{cl}(X,\Ko)$ we may write $Q=Q_0+Q_1$ where $Q_1\in \Psi^{-1,p}_{cl}(X,\Ko)$ and $Q_0$ is defined from a homogeneous symbol. Clearly $Q\in \ker \sigma$ if and only if $Q_0\in \ker \sigma$. The last statement holds if and only if some full symbol $\sigma_\mathfrak{U}(Q_0,S)$ vanishes for large $|\xi|$. Using the decomposition \eqref{decompconstschatt} of $\sigma_\mathfrak{U}(Q_0,S)$, it follows that $\sigma^0_\mathfrak{U}(Q_0,S)=0$ and $\sigma^1_\mathfrak{U}(Q_0,S)\in C^{N,\infty}_c(T^*X,\ellL^p)$ satisfies the estimates \eqref{secondest}-\eqref{thirdest} so $\sigma^1_\mathfrak{U}(Q_0,S)\in \Sigma^{0,p}_{cc}(X,\Ko)$ and $Q_0\in \Psi^{0,p}_{cc}(X,\Ko)$.
\end{proof}

The symbols however have a rather simple structure, and the following proposition follows from the decomposition \eqref{decompconstschatt}:

\begin{prop}
The principal symbol algebra satisfies that
\[\Sigma^{p}(X,\Ko)\subseteq C^\infty(X,\Bo(\He))+C^{N,\infty}(S^*X,\ellL^p(\He))\subseteq C^\infty(S^*X,\Bo(\He)),\]
where we extend elements of $C^\infty(X,\Bo(\He))$ to constant functions on the fiber of $\pi:S^*X\to X$. Any element of $C^\infty(X,\Bo(\He))+C^{N,\infty}(S^*X,\ellL^p(\He))$ that in local coordinates satisfies the estimates \eqref{secondest} and \eqref{thirdest} extends to an element of $\Sigma^{0,p}(X,\Ko)+C^{N,\infty}_c(T^*X,\ellL^p)$.
\end{prop}

We will use the notation $\Sigma(X,\Ko):=C(X,\Bo(\He))+C(S^*X,\Ko(\He))$, which is the $C^*$-closure of $\Sigma^p(X,\Ko)$ in the pointwise operator norm for any $p\geq 1$. Recall that an embedding $\mathcal{A}\hookrightarrow A$ of bornological algebras is called isoradial if the embedding preserves spectral radius of bounded subsets, see Definition $2.21$ and Definition $2.48$ of \cite{cumero}. In particular, a dense isoradial embedding preserves spectrum by Lemma $2.50$ of \cite{cumero} and induces an isomorphism on $K$-theory by Theorem $2.60$ of \cite{cumero}. We equip $\Sigma^{p}(X,\Ko)$ with the bornology induced from the inclusion $\Sigma^{p}(X,\Ko)\subseteq C^\infty(S^*X,\Bo(\He))$ and the seminorms induced from the proportionality constants in \eqref{secondest}-\eqref{thirdest}. 

\begin{prop}
\label{isoembedd}
The inclusion $\Sigma^{p}(X,\Ko)\subseteq \Sigma(X,\Ko)$ is dense and isoradial. So if $a\in\Sigma^{p}(X,\Ko)$ is invertible in $\Sigma(X,\Ko)$ it is also invertible in $\Sigma^{p}(X,\Ko)$  and $K_*(\Sigma^{p}(X,\Ko))\cong K_*(\Sigma(X,\Ko))$. 
\end{prop}

\begin{proof}
It is straight forward that $\Sigma^{p}(X,\Ko)$ is dense in $\Sigma(X,\Ko)$. To prove that the embedding is isoradial we follow the idea of Example $2.18$ and Lemma $2.49$ of \cite{cumero}. Assume that $S\subseteq \Sigma^{p}(X,\Ko)$ is contained in the closed ball of radius $r<1$ in $\Sigma(X,\Ko)$, then Lemma $2.49$ of \cite{cumero} implies that $\Sigma^{p}(X,\Ko)\subseteq \Sigma(X,\Ko)$ is isoradial if $S$ is power-bounded in $\Sigma^{p}(X,\Ko)$, see Definition $2.16$ of \cite{cumero}. To prove that $S$ is power-bounded we must show that if $f_1,f_2,\ldots, f_k\in S$, the semi-norms of $f_1f_2\cdots f_k$ are bounded by semi-norms of $f_1,f_2,\ldots, f_k$. The assumptions on $S$ implies that $\|f_i\|_{\Sigma(X,\Ko)}\leq r<1$. Since $\ellL^p$ is a symmetric operator ideal in $\Bo$ this assertion and the fact that $S$ is power-bounded in $\Sigma^{p}(X,\Ko)$ follows in the same fashion as in Example $2.18$ of \cite{cumero}.
\end{proof}

We will now describe the elliptic operator-valued pseudo-differential operators. We say that an operator-valued pseudo-differential operator $Q$ is elliptic if there is a smoothing operator $S$ such that the full symbol $\sigma_\mathfrak{U}(Q,S)$ takes invertible values outside a compact subset of $T^*X$. A classical operator-valued pseudo-differential operator $Q$ is clearly elliptic if and only if $\sigma(Q)$ takes values in the invertible operators. 

\begin{prop}
An operator-valued pseudo-differential operator is elliptic if and only if it is invertible up to a smoothing operator. 
\end{prop}

\begin{proof}
If an operator-valued pseudo-differential operator $Q$ with a full symbol $a$ is invertible up to a smoothing operator, a full symbol $r$ of an inverse up to smoothing terms $R$ satisfies $1-ra,1-ar\in \Sigma^{-1,p}(X,\Ko)$. We define $r_k:= \sum_{j=0}^{k-1} (1-ra)^jr$. For sufficiently large $k$, $r_k$ is an inverse of $a$ up to a term that defines a smoothing operator. So there is a full symbol of $Q$ that is invertible outside a compact set. 

To prove the converse, assume that $Q\in \Psi^{0,p}(X,\Ko)$ is elliptic. Let us choose a full symbol $a:=\sigma_\mathfrak{U}(Q,S)$ that is invertible outside a compact subset of $T^*X$ with inverse $r$ corresponding to a parametrix $R$. In particular, it follows from \eqref{prodform} that there are smoothing operators $S_1,S_2$ such that $\sigma_\mathfrak{U}(1-QR,S_1),\sigma_\mathfrak{U}(1-RQ,S_2)\in \Sigma^{-1,p}(X,\Ko)+\Sigma^{m,p}_{cc}(X,\Ko)$. Operator-valued pseudo-differential operators of order $(-1,p)$ are in some Schatten class and compactly defined symbols of order $(0,p)$ operators are also in some Schatten class, due to the decomposition \eqref{decompconstschatt}. Hence, for a large $q>0$ we have that $1-QR,1-RQ\in \ellL^q(L^2(X,\He))$. Take an integer $k>q$ and define $R_k:=\sum_{j=0}^{k-1} (1-RQ)^jR$. Since $1-QR_k,1-R_kQ\in \ellL^1(L^2(X,\He))$, $Q$ is invertible up to a smoothing operator.
\end{proof}

It follows from that $\Psi^{m,p}(X,\Ko)\subseteq \ellL^q(L^2(X,\He))$ for some $q$ whenever $m<0$ that an operator-valued pseudo-differential operator is elliptic if and only if it is Fredholm. We shall now review their index theory as described in \cite{grigtwo}. To describe the index theory for general operator-valued pseudo-differential operators we need to obtain a homotopy from the elliptic operator-valued pseudo-differential operators to classical elliptic operator-valued pseudo-differential operators. The case of classical pseudo-differential operators is easier since the index of a classical elliptic operator will only depend on its principal symbol. The following Lemma follows in a similar fashion as Theorem $19.2.3$ of \cite{hormtre}:

\begin{lem}
\label{pathlem}
If $Q\in \Psi^{0,p}(X,\Ko)$ is elliptic, then modulo Schatten class terms, there is a norm-continuous path $(Q_t)_{t\in [0,1]}\subseteq \Psi^{0,p}(X,\Ko)$ of elliptic operator-valued pseudo-differential operators such that $Q_0=Q$ and $Q_1$ is classical.
\end{lem}

\begin{proof}
If $Q$ is elliptic, we may write $Q=\sum_\alpha u_{\alpha}\left[ a_\alpha(x,\partial)\right]u_{\alpha}^{-1} +S$ as above in \eqref{locdecomp} in such a way that $\sigma_\mathfrak{U}(Q,S)$ takes invertible values outside a ball subbundle of $T^*X$ of some radius $s$. Using the formula \eqref{intedecomp} we can for each $\alpha$ decompose $a_\alpha(x,\xi)=a_{\alpha,0}(x)+a_{\alpha,1}(x,\xi)$. Let $\chi\in C^\infty(\R_+)$ be a smooth cutoff satisfying $\chi(u)=0$ near $u=0$ and $\chi(u)=1$ for $u\geq 1$. Define $(a_{\alpha}^t)_{t\in [0,1]}\subseteq \Sigma^{0,p}_c(U_\alpha,\Ko)+C^{N,\infty}_c(T^*X,\ellL^p)$ by 
\[a_\alpha^t(x,\xi):=a_{\alpha,0}(x)+\chi(|\xi|) a_{\alpha,1}(x,s^t|\xi|^{-t}\xi)+(1-\chi(|\xi|) )a_{\alpha,1}(x,\xi).\]
The path $t\mapsto a^t_\alpha$ is smooth, $a^0_\alpha=a_\alpha$ and $a^1_\alpha$ is homogeneous. Observe that $a_\alpha^t\notin \Sigma^{0,p}_c(U_\alpha,\Ko)$ since it violates \eqref{secondest}-\eqref{thirdest} in a compact set. Compactly supported symbols taking values in a Schatten class defines Schatten class pseudo-differential operators and therefore this path of symbols defines a path of operator-valued pseudo-differential operators modulo Schatten class terms. Furthermore, if $R$ is a parametrix of $Q$, with $R=\sum_\alpha u_{\alpha}\left[r_\alpha(x,\partial)\right]u_{\alpha}^{-1}+S'$, then we can in the same way as above define $r^t_\alpha$. Since $t\mapsto a^t_\alpha$ is smooth the path of pseudo-differential operators defined by
\[Q_t:=\sum_\alpha u_\alpha \left[a_\alpha^t(x,\partial)\right] u_\alpha^{-1} +S,\]
is norm-continuous with a norm-continuous path of parametrices $(R_t)_{t\in [0,1]}$, so the path satisfies the properties in the statement of the Lemma.
\end{proof}

\subsection{$K$-theory of the symbol algebra} 

The index theory of operator-valued pseudo-differential operators does in much behave as index theory on finite-dimensional fibers. One large difference is the $K$-theory of the symbol algebra. The full symbol $a$ of an elliptic operator-valued pseudo-differential operator is in fact a continuous function $a:T^*X\to \mathcal{F}_0$, where $\mathcal{F}_0$ denotes the space of Fredholm operators of index $0$, that takes invertible values outside a compact set. This construction gives an element of $\Xi[a]\in K^0(T^*X)$ by the Atiyah-J\"anich theorem. We will in this subsection look closer at the symbol class in $K^0(T^*X)$ using methods from operator algebras and show that $\Xi$ is an isomorphism on $K$-theory and the index of an operator-valued pseudo-differential operator is given by $\ind_X\Xi[a]$, where $\ind_X:K^*(T^*X)\to \Z$ denotes the index mapping.

Motivated by Proposition \ref{pathlem} we define the symbol class $[Q]\in K_1(\Sigma(X,\Ko))$ of an elliptic operator-valued pseudo-differential operator $Q$ as the class $[\sigma(Q_1)]$. The index of an operator-valued pseudo-differential operator will only depend on the symbol class and coincides with the index pairing with the $K$-homology class $\Psi^X\in K^1(\Sigma(X,\Ko))$ defined by \eqref{psix}. A priori, \eqref{psix} do only define an extension class. We will not prove that this extension class comes from a $K$-homology class until section \ref{pair} where we prove it in the more general setting of projective symbols. Let us calculate the $K$-theory of the principal symbol algebra.

\begin{lem}
\label{ktheorytrivdd}
If $\He$ is infinite-dimensional, the association $a\mapsto \Xi[a]$ defines an isomorphism 
\[\Xi:K_*(\Sigma(X,\Ko))\to K^{*+1}(T^*X).\]
\end{lem}

\begin{proof}
Let $\Co(\He)$ denote the Calkin algebra. Consider the short exact sequences
\begin{align}
\label{seo}
&0\to C(S^*X,\Ko(\He))\to \Sigma(X,\Ko)\to C(X,\mathcal{C}(\He))\to 0,\\
\label{set}
&0\to C(X,\Ko(\He))\to C(X,\Bo(\He))\to C(X,\mathcal{C}(\He))\to 0.
\end{align}
The index mapping $K_*(C(X,\mathcal{C}(\He)))\to K_{*+1}(C(X,\Ko(\He)))$ induced by \eqref{set} is an isomorphism since the K\"unneth theorem implies that 
\[K_*(C(X,\Bo(\He)))=0,\] 
while $K_*(\Bo(\He))=0$ if $\He$ is infinite-dimensional. Therefore \eqref{seo} gives a six term exact sequence which after using Morita invariance of $K$-theory looks like:
\[
\begin{CD}
K^0(S^*X)  @>>> K_0(\Sigma(X,\Ko))@>>>K^1(X)\\
@A\pi^*AA@. @VV\pi^*V  \\
K^{0}(X)@<<<  K_1(\Sigma(X,\Ko))@<<< K^1(S^*X) \\
\end{CD}. \]
On the other hand, consider the short exact sequence
\[0\to C_0(T^*X)\to C(\bar{B}^*X)\to C(S^*X)\to 0.\]
After taking $K$-theory of this short exact sequence and using that $\pi^*:C(X)\to C(\bar{B}^*X)$ defines a homotopy equivalence we arrive at the six-term exact sequence 
\[
\begin{CD}
K^0(S^*X)  @>>> K^1(T^*X)@>>>K^1(X)\\
@A\pi^*AA@. @VV\pi^*V  \\
K^{0}(X)@<<<  K^0(T^*X)@<<< K^1(S^*X) \\
\end{CD}, \]
where the mapping $K^*(T^*X)\to K^*(X)$ is the restriction mapping. A diagram chase and the five lemma implies that $\Xi$ is an isomorphism $K_*(\Sigma(X,\Ko))\cong K^{*+1}(T^*X)$. 
\end{proof}

In fact, the isomorphism $\Xi$ can be made explicit using $KK$-theory in the same that one associates a difference class in $K^0(T^*X)$ to an elliptic pseudo-differential operator. We shall return to this observation for projective symbols. Observe that the $K$-theory of the algebra of principal symbols is quite different when the fibers are of infinite dimension. In finite dimension the Calkin algebra is trivial so the $K$-theory of the principal symbol algebra is $K^*(S^*X)$, thus it also contains non-trivial elements that are constant on the fibers. However, as far as index theory is concerned this will not play a role since the pseudo-differential extension \eqref{psix} quantize the base of $T^*X$ in a trivial way.

In the situation of a fiber bundle $Z\to X$ with compact fiber and compact base the map $\Xi$ acts as a fiberwise index mapping for elliptic operators. There is, as noted above, an embedding of the symbols on $Z$ into the operator-valued symbols on $X$. As in Lemma \ref{ktheorytrivdd}, a symbol on $Z$ that defines an elliptic operator-valued symbol on $X$ is a section $T^*X\to \mathcal{F}_0(L^2(Z|X))$ that takes invertible values outside a compact subset, where $\mathcal{F}_0(L^2(Z|X))$ denotes the bundle of Fredholm operators of index $0$.  On the other hand, a fiberwise elliptic operator on $Z\to X$ defines a section $X\to \Fg(L^2(Z|X))$, the bundle of Fredholm operators on $L^2(Z|X)$, which in term determines a class in $K^0(X)$.

\subsection{The Chern character of a symbol}

On an abstract level, we can define the Chern character on the principal symbol algebra $\widetilde{\ch}:K_*(\Sigma(X,\Ko))\to H_{c}^{*+1}(T^*X)$ by representing a $K$-theory class of $\Sigma(X,\Ko)$ by a $K$-theory class on $T^*X$. We will describe how to calculate the Chern character more explicitly using the dense isoradial $*$-subalgebras $\Sigma^p(X,\Ko)$ by modifying symbols to take values in the Schatten class operators. 

\begin{lem}
\label{schattenrep}
For any invertible $a\in \Sigma^{p}(X,\Ko(\He))$ there is an invertible $a_0\in C^\infty(X,\Bo(\He\oplus \He))$ such that 
\[(a\oplus 1)-\pi^*a_0\in C^{N,\infty}(S^*X,\ellL^p(\He\oplus \He))\cap \Sigma^p(X,\Ko(\He\oplus \He)).\]
\end{lem}

The proof of this Lemma follows the same lines as that of Proposition $4.3$ of \cite{grigtwo}, but we formulate the result in a different fashion that will prove the same result for projective symbols mutatis mutandis. If the Euler characteristic of $X$ vanish, the situation is much simpler, the cosphere bundle $S^*X\to X$ admits a global section $\eta$ and $a_0:=\eta^*a\oplus 1$ satisfies the conditions of the Lemma.

\begin{proof} 
We define $\bar{a}:=\pi_* a\in C^\infty(X,\Bo(\He))$. We have that $\bar{a}$ takes values in $\mathcal{F}_0$ and so does $\bar{a}\oplus 1\in C^\infty(X,\Bo(\He\oplus \He))$. Furthermore, 
\[a-\bar{a}\in C^{N,\infty}(S^*X,\ellL^p(\He))\cap \Sigma^p(X,\Ko(\He)).\] 
After adding a finite-rank valued function to $\bar{a}\oplus 1$ we obtain a function $a_0$ that takes surjective values, so $a_0$ takes invertible values. 
\end{proof}

A consequence of Lemma \ref{schattenrep} is that the construction of the Chern character of $a$ reduces to constructing the Chern character of $(\pi^*a_0)^{-1}a\in 1+C^{N,\infty}(S^*X,\ellL^p(\He))\cap \Sigma^p(X,\Ko(\He))$. This follows from that $[(\pi^*a_0)^{-1}a]=[a]-\pi^*[a_0]=[a]$ in $K_1(\Sigma(X,\Ko))$ since $K_*(C(X,\Bo))=0$. The construction of the Chern character in the case $p=1$ is a straight-forward generalization of the classical Chern character. For $p>1$ the calculation of the Chern character becomes more complicated in general. One must resort to a different approach using a smaller subalgebra of the principal symbol algebra and regularize the Chern character using the suspension operator in cyclic homology. The construction in this case is done locally in \cite{grigone} and the global case in \cite{grigtwo}. We refer the reader to these papers for the explicit construction.

There are in some cases possible to calculate the Chern characters more directly. If $\rho$ is a differential form on a fiber bundle we denote its component that has degree $k$ in the fiber directions by $[\rho]_k$. We let $\mathfrak{d}:H^*(S^*X)\to H^{*+1}_c(T^*X)$ denote the boundary mapping.

\begin{prop}
If $a\in \Sigma^p(X,\Ko)$ is invertible and $p\leq n-1$ then  
\[\widetilde{\ch}[a]_n=\mathfrak{d} \left(\sum_{k=0}^{n-1}\tra\left[(a^{-1}\rd a)^{2k+1}\right]_{n-1}\right).\]
\end{prop}

All the terms in the Proposition are well defined since any differential of $a$ in the fiber direction is $n-1$-summable. This proposition follows from homotopy invariance of the Chern character which by Lemma \ref{schattenrep} allows us to assume that $a\in 1+C^{N,\infty}(T^*X, \ellL^p(\He))$.

\subsection{The index theorem}

We end this chapter by stating an Atiyah-Singer type result for the operator-valued pseudo-differential operators from \cite{grigtwo}. The proof is by standard methods, see for instance \cite{atiyahsingerI} and \cite{atiyahsingerIII}.

\begin{sats}
\label{trivddindex}
If $Q$ is an elliptic operator-valued pseudo-differential operator of order $0$ on the closed manifold $X$ then 
\[\ind(Q)=\int_{T^*X} \widetilde{\ch}[Q]\wedge Td(X)\]
In particular, if the principal symbol $a$ of $Q$ is an element of $\Sigma^p(X,\Ko)$ for some $p\leq n-1$ then
\[\ind(Q)=\int_{S^*X} \sum_{k=n-1}^{2n-1}\tra\left[(a^{-1}\rd a)^{k}\right]_{n-1}\wedge \pi^*Td(X).\]
\end{sats}

\subsection{Example of free cocompact actions}
Let us give an example of how to construct operator-valued pseudo-differential operators on a compact quotient by a free discrete group action. The set up is that $X$ is a possibly non-compact smooth Riemannian manifold and the discrete group $\Gamma$ acts freely and cocompactly on $X$ by isometries. Since the $\Gamma$-action is free, $X\times_\Gamma \ell^2(\Gamma)\to X/\Gamma$ is a smooth Hilbert space bundle. A smooth section on this bundle is an equivariant function $X\to \ell^2(\Gamma)$ or compact-valued function $X\to \Ko(\ell^2(\Gamma))$. Observe that $L^2(X)\cong L^2(\Fg)\otimes \ell^2(\Gamma)\cong L^2(X\times_\Gamma \ell^2(\Gamma))$, where the later denotes the $L^2$-sections of $X\times \ell^2(\Gamma)$ and $\Fg$ denotes a fundamental domain. We will now write down a necessary and sufficient condition for a pseudo-differential operator on $X$ to define an operator-valued symbol on $X\times_\Gamma \ell^2(\Gamma)$ under the isomorphism $L^2(X)\cong L^2(X/\Gamma,X\times_\Gamma \ell^2(\Gamma))$.

The first step is that we with a symbol $a$ of order $m$ on $X$ can associate a function $\hat{a}_0\in C^\infty(T^*X|_\Fg,\Bo(\ell^2(\Gamma)))$ by setting $\hat{a}_0:=\sum_{\gamma\in \Gamma} \gamma^*a|_\Fg \delta_\gamma$, here the Dirac function $\delta_\gamma\in C_c(\Gamma)$ act on $\ell^2(\Gamma)$ by point-wise multiplication. Let $\hat{a}$ denote the equivariant extension of $\hat{a}_0$ to $T^*X$. Since the operators $\delta_\gamma$ are rank one projections onto an orthonormal basis of $\ell^2(\Gamma)$ it follows that 
\[\|\partial^\alpha_x\partial_\xi^\beta \hat{a}(x,\xi)\|_{\ellL^q(\ell^2(\Gamma))}=\sum_{\gamma\in \Gamma} |\partial_x^\alpha\partial_\xi^\beta \gamma^*a(x,\xi)|^q.\] 
Since $\hat{a}$ is equivariant it defines a smooth section $X/\Gamma\to \Bo(X\times_\Gamma \ell^2(\Gamma))$ and $\hat{a}\in \Sigma^{m,p}(X/\Gamma,\Ko(\ell^2(\Gamma)))$ if and only if 
\begin{align}
\label{gammafirstest}
\sum_{\gamma\in \Gamma} |\partial_x^\alpha\partial_\xi^\beta \gamma^*a(x,\xi)|^{p/(|\beta|-m)}&<\infty \quad \mbox{for} \quad m+p>|\beta|>m \quad\mbox{and}\\ 
\label{gammasecest}
\sum_{\gamma\in \Gamma} |\partial_x^\alpha\partial_\xi^\beta \gamma^*a(x,\xi)|&\lesssim (1+|\xi|)^{m+p-|\beta|} \quad \mbox{for} \quad |\beta|>m+p.
\end{align}
The function $\hat{a}$ automatically satisfies condition \eqref{firstest}. This construction of operator-valued symbols is but a consequence of the isomorphism $C_0(X)\ltimes \Gamma\cong C(X/\Gamma)\otimes \Ko$ coming from Green's imprimitivity theorem. The association $a\mapsto \hat{a}$ can be extended to non-local operators, i.e. operators on $L^2(X)$ of the form $\sum a_\gamma \gamma$. These types of operators have previously been studied in \cite{nasast} when $\Gamma$ is of polynomial growth and $X$ is compact but the only restriction on the $\Gamma$-action is that it embeds into a compact Lie group of diffeomorphisms of $X$. The symbol of such an operator is an element of $C^\infty(T^*X)\ltimes \Gamma$, the later is as a linear space given by $C^\infty(T^*X)\hat{\otimes} \mathcal{S}(\Gamma)$. Here $\mathcal{S}(\Gamma)$ denotes the Schwartz functions on $\Gamma$ and forms a Frechet algebra under convolution. We let $S^m_{\Gamma,p}(X)$ denote the Frechet space formed by elements of $S^m(X)$ satisfying \eqref{gammafirstest} and \eqref{gammasecest}. If $\Gamma$ is of polynomial growth acting freely and cocompactly on $X$ as above, then we can clearly extend $a\mapsto \hat{a}$ to a mapping $S^m_{\Gamma,p}(X)\hat{\otimes} \mathcal{S}(\Gamma)\to \Sigma^{m,p}(X/\Gamma,\Ko(\ell^2(\Gamma)))$. If we let $\lambda_\gamma\in \Bo(\ell^2(\Gamma))$ denote left translation by $\gamma$ the image of an element $a=\sum a_\gamma \gamma\in S^m_{\Gamma,p}(X)\hat{\otimes} \mathcal{S}(\Gamma)$ under this mapping is the equivariant extension of the $\Bo(\ell^2(\Gamma))$-valued function 
\[\sum_{\gamma\in \Gamma} \gamma^*a_{\gamma'}|_\Fg \delta_\gamma\lambda_{\gamma'}\in C^\infty(T^*X|_\Fg,\Bo(\ell^2(\Gamma))).\]

In the case that $\Gamma$ is of polynomial growth there is a simple sufficient condition implying \eqref{gammafirstest} and \eqref{gammasecest}. We can find a smooth function $v:X\to \R_+$ such that $v(x\gamma)\gtrsim \rd_\Gamma(\gamma,1)$ for all $x\in \Fg$, where $\rd_\Gamma$ denotes some choice of word metric on $\Gamma$. Assume that $a$ is a symbol of order $m$ on $X$ that satisfies 
\begin{equation}
\label{polyest}
|\partial^\alpha_x\partial^\beta_\xi a(x,\xi)|\lesssim (1+|\xi|+v(x)^\delta)^{m-|\beta|},
\end{equation}
for a $\delta>0$. If we denote the order of growth of $\Gamma$ by $N$ then it follows that $a\in S^m_{\Gamma,p}(X)$ for any $p>\max(1,N/\delta)$. A very simple example of such a symbol is if $D=\sum a_j\partial_j+a_0$ is a first order differential operator on $X$ and $|a_0|\gtrsim v^\delta$, then $\sigma(D)/(1+\sigma(D)^*\sigma(D))$ satisfies \eqref{polyest} with $m=0$.

For operator-valued pseudo-differential operators of this type the calculation of their Chern character for $p\leq n-1$ can be done explicitly. Recall that $[\rho]_k$ denotes the part of a form or a cohomology class $\rho$ that contains $k$ differentials in the fiber direction.

\begin{lem}
Assume that the symbol $a=\sum_{\gamma\in \Gamma}a_\gamma \gamma\in S^0_{\Gamma,p}(X)\hat{\otimes}\mathcal{S}(\Gamma)$ is unitary and $p\leq n-1$. Then we have that 
\[\widetilde{\ch}[\hat{a}]_{n}=\mathfrak{d}\left(\sum_{k=n}^{2n} \ch_{k,n-1}[\hat{a}]\right)\in H^*_c(T^*(X/\Gamma)),\]
and the terms $\ch_{k,n-1}[\hat{a}]$ are defined from the form defined as the following absolutely convergent sum:
\[\sum_{(\gamma_{i,j})\in \Gamma^{3k}_{rel}}\left[\bigwedge_{j=1}^k\overline{\gamma_{1,j}^*a_{\gamma_{2,j}}}\rd \left(\gamma_{1,j}^*a_{\gamma_{3,j}}\right)\right]_{n-1},\]
where the subset $\Gamma^{3k}_{rel}\subseteq \Gamma^{3k}$ is defined as all sequences $(\gamma_{1,j},\gamma_{2,j},\gamma_{3,j})_{j=1,\ldots, k}$ such that
\[\gamma_{3,j}^{-1}\gamma_{1,j}=\gamma^{-1}_{2,j+1}\gamma_{1,j+1}\quad\mbox{and}\quad \gamma_{2,1}^{-1}\gamma_{3,1}\gamma_{2,2}^{-1}\gamma_{3,2}\cdots \gamma_{2,k}^{-1}\gamma_{3,k}=1.\]
\end{lem}

\begin{proof}
That the sum is absolutely convergent follows from that the factors in the summands are polynomially decaying in $\gamma_{2j}$ and $\gamma_{3j}$ and $n-1$-summable in $\gamma_{1j}$. The calculation of the Chern character goes as follows; we have that 
\begin{align*}
\hat{a}^*\rd\hat{a}=\sum_{\gamma_1,\gamma_2,\gamma_3,\gamma_4\in \Gamma} \overline{\gamma_1^*a_{\gamma_2}}\gamma_3^*\rd a_{\gamma_4}\delta_{\gamma_2^{-1}\gamma_1}\lambda_{\gamma_2^{-1}}\delta_{\gamma_3}\lambda_{\gamma_4}=\\
=\sum_{\gamma_1,\gamma_2,\gamma_3\in \Gamma} \overline{\gamma_1^*a_{\gamma_2}}\gamma_1^*\rd a_{\gamma_3}\delta_{\gamma_2^{-1}\gamma_1}\lambda_{\gamma_2^{-1}\gamma_3}.
\end{align*}
It follows that
\[(\hat{a}^*\rd\hat{a})^k=\sum_{(\gamma_{i,j})\in \Gamma^{3k}}\bigwedge_{j=1}^k\overline{\gamma_{1,j}^*a_{\gamma_{2,j}}}\rd \left(\gamma_{1,j}^*a_{\gamma_{3,j}}\delta_{\gamma_{2,j}^{-1}\gamma_{1,j}}\lambda_{\gamma_{2,j}^{-1}\gamma_{3,j}}\right).\]
However, when taking the trace only the terms with $(\gamma_{i,j})\in \Gamma^{3k}_{rel}$ give a non-zero contribution. 
\end{proof}

\Large
\section{The symbol calculus on Azumaya bundles}
\label{sectwo}
\normalsize

In this section we will study a certain algebra of symbols on Azumaya bundles, symbols which we will associate pseudo-differential operators with in the next section. Along the way we will shortly recall some known results about continuous trace algebras, twisted $K$-theory and the corresponding Chern characters. 

In the previous section when we considered the situation of operator-valued pseudo-differential operators on an infinite-dimensional fiber, Kuiper's theorem revealed that trivial Hilbert space bundles contains the same topological invariants as diffeologically non-trivial Hilbert space bundles. Matters change drastically when the symbols of the pseudo-differential operators have a non-trivial global structure. The group of automorphisms of the compact operators is the projective unitary group, with compact-open or norm topology, whose classifying space in infinite dimension has the homotopy type of a $K(\Z,3)$, which will allow for topological invariants that can not exist on a finite-dimensional fiber. 

We will pick up the slack from the previous section and address the technical aspect of the compact-open topology on the structure group. On the level of $C^*$-algebras this will not affect much. However we know of no dense subalgebra of smooth symbols in this setting and no way of explicitly calculating Chern character other than using the spectral sequence of \cite{atiyahsegalcohom}. In the case of a norm smooth Azumaya bundle one can calculate the Chern character explicitly using \cite{mathstev}. For a brief overview of twisted $K$-theory we refer to \cite{rosentexas}.\\

Recall that a continuous trace algebra is a $C^*$-algebra $A$ whose spectrum $\hat{A}$ is Hausdorff and such that $A$ has enough locally rank one projections, see more in \cite{blackadaropalg}, \cite{dixmierdouady} or \cite{rosentexas}. If the spectrum $X:=\hat{A}$ of a stable separable continuous trace algebra $A$ is locally compact, there is a Hilbert space $\He$ and a locally trivial bundle $\Ko_A\to \hat{A}$ with fiber $\Ko(\He)$ such that $A\cong C_0(X,\Ko_A)$. We will call an algebra bundle $\Ko_A\to X$ with fiber being compact operators on a Hilbert space an Azumaya bundle over $X$. Since $\Ko_A$ is locally trivial it comes from a principal $PU(\He)$-bundle $P\to X$. Here 
\[PU(\He):=U(\He)/S^1=Aut(\Ko(\He)).\]

A principal $PU(\He)$-bundle $P\to X$ is, in a fine enough cover, described by a Cech cocycle $(c_{\alpha\beta})\in Z^1((U_\alpha),PU(\He))$, for some cover $(U_\alpha)$ of $X$. For simplicity we will always assume that the cover $(U_\alpha)$ is fine enough and the elements $c_{\alpha\beta}$ are already lifted to functions $c_{\alpha\beta}:U_\alpha\cap U_\beta\to U(\He)$. With this choice of Cech cocycle there is an associated continuous trace algebra
\[A:=\{(a_\alpha)\in \oplus_\alpha C(U_\alpha,\Ko(\He)): a_\alpha|_{U_\alpha\cap U_\beta}c_{\alpha\beta}=c_{\alpha\beta}a_\beta|_{U_\alpha\cap U_\beta}\}.\]
The Azumaya bundle associated with this continuous trace algebra is $\Ko(P):=P\times_{PU(\He)}\Ko$ and $A=C_0(X,\Ko(P))$. The algebra of sections of an Azumaya bundle with finite-dimensional fiber is in fact an Azumaya algebra over $C_0(X)$ motivating the terminology of infinite-dimensional Azumaya bundles. 

The opposite principal $PU(\He)$-bundle $P^{op}$ is the bundle defined by the Cech cocycle $\left(\overline{c_{\alpha\beta}}\right)$, with respect to a fixed real structure on $\He$. An alternative way is to define $P^{op}$ as a principal $PU(\He)^{op}$-bundle defined from the cocycle $(c_{\alpha\beta}^*)\in Z^1((U_\alpha),PU(\He)^{op})$. These two constructions coincide in the sense that a real structure defines an isomorphism $\Ko^{op}\cong \Ko$, $k\mapsto \bar{k}^*$ which defines an isomorphism $PU(\He)\cong PU(\He)^{op}$. 

The two cases that $\He$ is of finite respectively infinite dimension are quite different. We will in this paper only consider the case when $\He$ is of infinite dimension. In this case, $U(\He)$ is contractible in its compact-open topology by Kuiper's theorem, and therefore also in norm topology. It follows that its classifying space $BPU(\He)$ is a $K(\Z,3)$ and $H^1(X,PU(\He))\cong H^3(X,\Z)$ both in compact-open and norm topology. Therefore the image $\delta(P)$ of the Cech cocycle $(c_{\alpha\beta})$ under this isomorphism classifies the bundle up to isomorphism. In particular, one can always find an isomorphism class of a principal $PU(\He)$-bundle in compact-open topology that is norm continuous. The invariant $\delta(P)$ is called the Dixmier-Douady invariant of the principal $PU(\He)$-bundle $P$ or the associated Azumaya bundles. Without going into too much details, let us mention that for finite-dimensional fibers the Dixmier-Douady invariant is a torsion class in $H^3(X,\Z)$ that only classifies Morita equivalence classes of such bundles. 

Any ideal $\I\subseteq \Bo(\He)$ defines a locally trivial bundle $\I(P):=P\times_{PU(\He)}\I \to X$ with fiber $\I$. The bundle $\I(P)$ is a continuous bundle whenever $\I$ has a topology on which $PU(\He)$ acts continuously. The cases of interrest for us are $\I=\Bo(\He)$, $\I=\Ko(\He)$ and the Schatten class operators $\I=\ellL^p(\He)$, but of course other ideals might be useful in other index theories. Observe that $PU(\He)$ with norm topology acts norm continuously on $\Bo(\He)$, $\Ko(\He)$ and $\ellL^p(\He)$. In compact-open topology $PU(\He)$ acts continuously on $\Bo(\He)$ only in its strong topology but on $\Ko(\He)$ and $\ellL^p(\He)$ in their norm topology, see Proposition A$1.1$ of \cite{atiyahsegaltwist}. The group $PU(\He)$ equipped with norm topology can in fact be made into a Banach-Lie group, it is a closed Lie subgroup of $U(\He\otimes \overline{\He})$. Therefore it makes sense to speak of smooth principal $PU(\He)$-bundles. In this paper we will only consider two classes of principal $PU(\He)$-bundles, those that are smooth and those equipped with compact-open topology on their fibers. A middle thing between the two is as far as the techniques in this paper are concerned too impractical to perform general calculations and a too restrictive class to be able to use in any examples.

If $P$ and $P'$ are two principal $PU(\He)$-bundles over $X$ the fiber product $P\times _{X}P'$ is a principal $PU(\He)\times PU(\He)$-bundle. On the level of infinite-dimensional Azumaya bundles over $X$ we have that  $\Ko(P\times_X P')=\Ko(P)\otimes_X\Ko(P')$, which is given by taking the fiberwise $C^*$-tensor product. If $\He$ is infinite-dimensional we can choose an isomorphism $\He\otimes\He\cong \He$ which defines a homomorphism $U(\He)\times U(\He)\to U(\He)$ that acts as multiplication on the centers. Hence, on the level of cohomology, this mapping induce the addition in $H^3(X,\Z)$, i.e. the Dixmier-Douady invariant of $P\times_X P'$ under the isomorphism $\He\otimes\He\cong \He$ is given by $\delta(P)+\delta(P')$. Clearly, since $P\times_X P^{op}$ is a trivializable principal bundle $\delta(P^{op})=-\delta(P)$. On the level of Cech cocycles this implies that there is a Cech cochain $(\tilde{c}_\alpha)$ with values in $U(\He\otimes\He)$ such that $c_{\alpha\beta}\otimes \overline{c_{\alpha\beta}}=\tilde{c}_\alpha^*\tilde{c}_\beta$.

\subsection{Example: $T$-duality} 
\label{tdualbundle}
An interesting way to construct a continuous trace algebra that we shall see fits very well together with our symbol calculus is via $T$-duality. The $T$-duality construction produces a duality between continuous trace algebras over the total space of a principal $S^1$-bundle and its crossed product by $\R$. Mathematically, $T$-duality is described by the Raeburn-Rosenberg theorem from \cite{raeros}. If $\pi_Z:Z\to X$ is a principal $S^1$-bundle and $P\to Z$ a principal $PU(\He)$-bundle, there is a principal $S^1$-bundle $\pi_Z^T:Z^T\to X$ and a principal $PU(\He)$-bundle $P^T\to Z^T$ such that $C(Z,\Ko(P))\ltimes \R\cong C(Z^T,\Ko(P^T))$. The $T$-duals $Z^T$ and $\delta(P^T)$ are determined by 
\[(\pi_Z)_*\delta(P)=c_1(Z^T) \quad{and}\quad (\pi_Z^T)_*\delta(P^T)=c_1(Z).\]
Observe that $\R$ do a priori only act fiberwise on $Z$, but the lifting lemma from \cite{raeros} for $\R$-actions on continuous trace algebras allows us to lift the $\R$-action to $C(Z,\Ko(P))$. Existence of the Connes-Thom isomorphism implies that $C(Z^T,\Ko(P^T))$ has the same $K$-theory as $C(Z,\Ko(P))$ up to a degree shift. We will focus on the case that $P$ is trivial and consider the continuous trace algebra $C(Z)\ltimes \R$ over $X\times S^1$.

In order to construct symbols we will need to understand the local structure of such a continuous trace algebra well. First, let us consider the most trivial case, that is, $Z=S^1$ as a bundle over a point. If $\theta\in \R/\Z=S^1$ we define 
\[\He_\theta:=\{f:\R\to \C:\;f(t+1)=\e^{2\pi i\theta}f(t),\; \int_0^1|f|^2\rd t<\infty\}.\]
Letting the parameter $\theta\in S^1$ vary, we may view the collection $(\He_\theta)$ as a continuous field of Hilbert spaces $\mathbbm{H}\to S^1$ with fiber being $L^2(0,1)$. In fact, a trivialization is given by simply restricting an element of $\He_\theta$ to $(0,1)$, the inverse is given by extending to $\R$ according to $f(t+1)=\e^{2\pi i\theta}f(t)$. Let $\lambda_\theta:C(S^1)\ltimes \R\to \Ko(\He_\theta)$ be defined on $a\in C_c(S^1\times \R)$ by 
\[\lambda_\theta(a)f(t):=\int_\R a(t,t-r)f(r)\rd r=\int_0^1 \left(\sum_{k\in \Z} a(t,t-k-r)\e^{2\pi i k\theta}\right) f(r)\rd r.\]
Since $a$ has compact support in its second variable, the sum is finite so the integral kernel of $\lambda_\theta(a)$ is continuous and therefore defines a compact operator. Therefore, $a\mapsto (\theta\mapsto \lambda_\theta(a))$, as a mapping $\lambda:C(S^1)\ltimes \R\to C(S^1,\Ko(\mathbbm{H}))$, produces an isomorphism $C(S^1)\ltimes \R\cong C(S^1,\Ko)$ after trivializing $\mathbbm{H}$. 

A remark that will play an important role in later chapters is that if $a\in C^\infty_c(S^1\times \R)$ then $\lambda_\theta(a)$ is a smoothing operator and $\lambda$ defines a morphism on the larger algebra $S^m(S^1)$ of pseudo-differential symbols on $S^1$ of order $m\leq 0$ after composing with the Fourier transform in the fiber direction. If $a\in \mathcal{S}'(T^*S^1)=\mathcal{S}'(S^1\times \R)$ we let $\Fg_2(a)$ denote the Fourier transform in the second variable, i.e. in the fiber direction. If $a\in S^m(S^1)$ then $\Fg_2(a)\in \mathcal{M}(C(S^1)\ltimes \R)$. Hence the mapping
\[\tilde{\lambda}(a):=\lambda(\Fg_2(a)),\]
is well defined since $\lambda$ extends to the multiplier algebra. This operator-valued section's action on elements in $\He_\theta$ can be described rather simply in terms of a pseudo-differential operator. Observe that if $a\in S^m(S^1)$ the operator $\lambda_0(a)\in \Bo(\He_0)=\Bo(L^2(S^1))$ is the pseudo-differential operator on $S^1$ defined from $a$. Let $U_\theta:\He_\theta\to \He_0=L^2(S^1)$ be the unitary defined from the measurable function $\theta\mapsto \e^{-i\theta}$. If we expand an $f\in \He_\theta$ in a Fourier series $f(t)=\sum_k c_k \e^{2\pi i (k+\theta)t}$ and apply $\tilde{\lambda}_\theta(a)$ one arrive at the expression
\begin{equation}
\label{thetatozero}
\tilde{\lambda}_\theta(a)f(t)=\sum_k a(t, 2\pi(k+\theta)) c_k\e^{2\pi i (k+\theta)t}=U_\theta^*\tilde{\lambda}_0(a_\theta)U_\theta f(t),
\end{equation}
where the symbol $a_\theta\in S^m(S^1)$ is defined as $a_\theta(t,u)=a(t,u+2\pi\theta)$. From this observation the next Proposition follows.

\begin{prop}
\label{lambdasymbols}
The mapping $\tilde{\lambda}:S^m(S^1)\to C^\infty(S^1,\Bo(\mathbbm{H}))$ is multiplicative up to lower order terms, i.e. $\tilde{\lambda}(ab)-\tilde{\lambda}(a)\tilde{\lambda}(b)\in \tilde{\lambda}(S^{m+m'-1}(S^1))$ for $a\in S^m(S^1)$ and $b\in S^{m'}(S^1)$.
\end{prop}

The mapping $\lambda$ has exactly the right equivariance properties to glue together over a principal $S^1$-bundle. In the general case, let $Z\to X$ be a principal $S^1$-bundle. We can choose a sufficiently fine cover $(U_\alpha)$ of $X$ and a Cech cocycle $(u_{\alpha\beta})\in Z^1((U_\alpha),S^1)$ defining $Z$. Therefore
\[C(Z)\cong \{(a_\alpha)\in \oplus C(U_\alpha\times S^1): a_\alpha(x,t)=a_\beta(x,t+u_{\alpha\beta}(x))\}.\]
In the same way, the crossed product with $\R$ can be considered as a gluing of the algebras $C(U_\alpha\times S^1)\ltimes \R$. So $C(Z)\ltimes \R$ is the norm closure of the algebra given by all elements in the convolution algebra $(a_\alpha)\in \oplus C_c(U_\alpha\times S^1\times \R)$ satisfying the gluing condition $a_\alpha(x,t,r)=a_\beta(x,t+u_{\alpha\beta}(x),r)$. We shall now see what happens to the gluing condition after applying the isomorphism $\id\otimes \lambda: C(U_\alpha\times S^1)\ltimes \R\cong C(U_\alpha\times S^1, \Ko(\pi_2^*\mathbbm{H}))$, where $\pi_2:U_\alpha\times S^1\to S^1$ denotes the projection onto the second coordinate. 

Choose a Cech cochain $(\tilde{u}_{\alpha\beta})\in C^1((U_\alpha),\R)$ that lifts $(u_{\alpha\beta})$, i.e. $u_{\alpha\beta}=\tilde{u}_{\alpha\beta}\mod \Z$. Define the section $c_{\alpha\beta}:U_\alpha\times S^1\cap U_ \beta\times S^1\to U(\pi_2^*\mathbbm{H})$ by 
\begin{equation}
\label{tdualcocycle}
c_{\alpha\beta}(x,\theta)f(t):=f(t+\tilde{u}_{\alpha\beta}(x)), \quad f\in\He_\theta.
\end{equation}
For $a\in C_c(U_\alpha\times S^1\times\R)$ the gluing condition on $\id\otimes \lambda(a)$ is exactly that for $f\in \He_\theta$ 
\begin{align}
\label{tdualcocyclecondition}
\left(\id\otimes\lambda(a_\alpha)(x,\theta)\right)f(t)&=\int _\R a_\alpha(x,t,t-r)f(r)\rd r=\\
&=\int_\R a_\beta(x,t+u_{\alpha\beta}(x),t-r)f(r)\rd r=\nonumber\\
&=\int_\R a_\beta(x,t+\tilde{u}_{\alpha\beta}(x),(t+\tilde{u}_{\alpha\beta}(x))-r)f(r-\tilde{u}_{\alpha\beta}(x))\rd r=\nonumber\\
&=c_{\alpha\beta}(x,\theta)\left(\id\otimes\lambda(a_\beta)(x,\theta)\right)c_{\alpha\beta}(x,\theta)^*f(t)\nonumber
\end{align}
It follows that the gluing condition on $(\id\otimes \lambda(a_\alpha))$ is exactly 
\[(\id\otimes \lambda(a_\alpha))(x,\theta)=c_{\alpha\beta}(x,\theta)\left(\id\otimes\lambda(a_\beta)(x,\theta)\right)c_{\alpha\beta}(x,\theta)^*\]
A consequence of this calculation is the identity 
\[c_{\alpha\beta}c_{\beta\gamma}c_{\gamma\alpha}(x,\theta)=\e^{2\pi i \theta(\tilde{u}_{\alpha\beta}(x)+\tilde{u}_{\beta\gamma}(x)+\tilde{u}_{\gamma\alpha}(x))}.\]
The right hand side defines an $S^1$-valued Cech $2$-cocycle on $X\times S^1$. The cohomology class' image under the Bockstein mapping $H^2(X\times S^1,S^1)\to H^3(X\times S^1,\Z)$ is by definition the Dixmier-Douady invariant of $C(Z)\ltimes \R$.

\subsection{Example: fiberwise magnetic translations}
\label{twistedgroupbundle}

This example is much motivated by \cite{mmett}. We will assume that $\Gamma$ is a discrete group acting properly on the complete Riemannian manifolds $Z$ and $X$ and that these manifolds fit into a smooth fiber bundle $M\to Z\xrightarrow{\pi_Z} X$ for a smooth compact manifold $M$. The manifolds $Z$ and $X$ are not assumed to be compact. We will make a rather strong assumption on this bundle; we will require that $\pi_Z$ is isometric, equivariant, admits a smooth global section $s:X\to Z$ and every $x\in X$ is contained in a $\Gamma$-invariant open set $U_x$ such that there is a $\Gamma$-equivariant isomorphism of fiber bundles $\pi^{-1}_Z(U_x)\cong U_x\times M$ for some $\Gamma$-action on $M$. The manifold $M$ is assumed to be simply connected. We will also make the rather strong assumption that the $\Gamma$-action on $X$ is properly discontinuous and cocompact. Consequently, the actions on $Z$ and $X$ are free. This assumption guarantees that there is a cover $(U_\alpha)$ of $X$ such that each $U_\alpha$ is $\Gamma$-invariant and that there are open sets $U_\alpha^0$ such that $\gamma U_\alpha^0\cap U_\alpha^0=\emptyset$ if $\gamma\neq 1$. We can write $U_\alpha=\Gamma U_\alpha^0$. We can choose this cover so that $U_\alpha^0$ and the connected components of $U_\alpha\cap U_{\beta}$ always are diffeomorphic to the unit ball. We will also fix a collection of $\Gamma$-equivariant diffeomorphisms $f_\alpha:\pi_Z^{-1}(U_\alpha)\to U_\alpha\times M$. 

From the local trivializations $f_\alpha$ we can construct a first order operator $D_M$ on differential forms $Z$ which acts as vertical differentiation. The operator $D_M$ is $\Gamma$-equivariant since each $f_\alpha$ is. We can split $T^*Z=\pi_Z^*T^*X\oplus T^*_MZ$ into a horizontal respectively vertical part. The extra data that we will choose, from which we can construct an Azumaya bundle on $X/\Gamma$, is a vertical form $\eta\in  C^\infty(Z, T^*_M Z)$ such that the $2$-form $D_M\eta\in C^\infty(T^*Z,\wedge^2 T^*_MZ)$ is $\Gamma$-invariant.

The $1$-form $\gamma^*\eta-\eta$ is in the kernel of $D_M$ since $D_M\eta$ is $\Gamma$-invariant. Since we have assumed that $M$ is simply connected there is a unique solution $\phi_{\gamma}\in C^\infty(Z)$ to the differential equation $\gamma^*\eta-\eta=D_M\phi_{\gamma}$ satisfying the condition $s^*\phi_{\gamma}=0$. We define a projective $\Gamma$-action on $C_b(X,L^2(Z|X))$ by 
\[T_{\gamma} g:=\e^{i\phi_{\gamma}}\gamma^*g.\]
That $\gamma\mapsto T_{\gamma}$ is a projective action follows from a straight-forward calculation leading to that the expression $\phi_{\gamma}+\gamma^*\phi_{\gamma'}-\phi_{\gamma\gamma'}$ is constant on the fiber. One has that $T_{\gamma} T_{\gamma'}= \varsigma(\gamma,\gamma') T_{\gamma\gamma'}$ where $\varsigma$ is the $C^\infty(X,U(1))$-valued $2$-cocycle on $\Gamma$ defined by
\[\varsigma(\gamma,\gamma'):=\e^{is^*\gamma^*\phi_{\gamma'}}.\]
Similarly to in \cite{mmett} we define a $\Gamma$-action on $C_b(X,\Ko(L^2(Z|X)))$ by the adjointing with $T_\gamma$. A prototypical example of a twisted equivariant (unbounded) operator-valued function on $X$ is the operator $D_\eta:=D_M+i\eta$ acting between suitable bundles of forms. Let us remark that the isomorphism class of this $\Gamma$-action only depends on $\eta$ modulo $\ker(D_M:C^\infty(Z,T^*_MZ)\to C^\infty(Z, \wedge^2T^*_MZ))$. Let $\mu:H^*(\Gamma,C^\infty(X,U(1)))\to H^{*+1}(X/\Gamma,\Z)$ denote the composition of the natural mapping $H^*(\Gamma,C^\infty(X,U(1)))\to H^*(X/\Gamma,U(1))$ with the isomorphism $H^*(X/\Gamma,U(1))\cong H^{*+1}(X/\Gamma,\Z)$.

\begin{prop}
There is a principal $PU(L^2(M))$-bundle $P_{Z,\eta}\to X/\Gamma$, continuous in the compact-open topology on $PU(L^2(M))$, such that $C(X/\Gamma,\Ko(P_{Z,\eta}))=C_b(X,\Ko(L^2(Z|X)))^\Gamma$ and $\delta(P_{Z,\eta})=\mu [\varsigma]$.
\end{prop}

\begin{proof}
First of all it is clear that $C_b(X,\Ko(L^2(Z|X)))^\Gamma$ is a continuous field of $C^*$-algebras over $X/\Gamma$. We must show that every point $x\in X/\Gamma$ has an open neighborhood over which this field is trivializable with fiber $\Ko(L^2(M))$. Let us take an $\alpha$ such that $x\in U_\alpha/\Gamma$, with $(U_\alpha)$ being a cover of the type described above. Since $\Gamma$ acts properly discontinuous on $X$, we have a $C_0(U_\alpha/\Gamma)$-linear isomorphism
\[C_b(U_\alpha,\Ko(L^2(Z|X)))^\Gamma\cong C_b(U_\alpha^0,\Ko(L^2(M))),\]
given simply by extending functions $U_\alpha^0\to\Ko(L^2(M))$ to $\Gamma$-invariant ones. This implies that the field is a locally trivial bundle and there is a corresponding frame bundle $P_{Z,\eta}\to X/\Gamma$. Since the action of $\Gamma$ is strictly continuous it is clear that $P_{Z,\eta}\to X/\Gamma$ is continuous in the compact-open topology. 
\end{proof}

\subsection{The algebra of projective symbols}

We will now proceed with defining the algebra of projective symbols. Following the notations of Section \ref{sectionone},  if $P\to X$ is a smooth principal $PU(\He)$-bundle defined from the Cech cocycle $(c_{\alpha\beta})$ we define the algebra of projective symbols of order $(m,p)$ as
\begin{equation}
\label{mpdef}
\Sigma^{m,p}(X,\Ko(P)):=\{(a_\alpha)\in \oplus_\alpha \Sigma^{m,p}(U_\alpha,\Ko(\He)): a_\alpha|_{U_\alpha\cap U_\beta}\pi^*c_{\alpha\beta}=\pi^*c_{\alpha\beta}a_\beta|_{U_\alpha\cap U_\beta}\}.
\end{equation}
We will equip this $*$-algebra with the bornology induced from the bornological algebras $\Sigma^{m,p}(U_\alpha,\Ko(\He))$. The Cech cocycle $\pi^*c_{\alpha\beta}$ only depend on the base coordinates so the estimates \eqref{firstest}--\eqref{thirdest} are invariant under applying transition functions. Therefore we can describe elements of $\Sigma^{m,p}(X,\Ko(P))$ as sections $a\in C^\infty(T^*X,\Bo(\pi^*P))$ satisfying the estimates \eqref{firstest}--\eqref{thirdest} in local coordinates. If $U\subseteq X$ is open we define $\Sigma^{m,p}_c(U,\Ko(P))$ as the elements of $\Sigma^{m,p}(X,\Ko(P))$ with support contained in $\pi^{-1}(U)$. If $P$ is trivializable over $U$ then this trivialization clearly extends to a trivialization $\Sigma^{m,p}_c(U,\Ko(P))\cong \Sigma^{m,p}_c(U,\Ko)$. We also define 
\begin{align*}
\Sigma^{m,p}_{cc}(X,\Ko(P)):&=\Sigma^{m,p}(X,\Ko(P))\cap C^\infty_c(T^*X,\Bo(\pi^*P))=\\
&=\Sigma^{m,p}(X,\Ko(P))\cap C^{N,\infty}_c(T^*X,\ellL^p(\pi^*P)).
\end{align*}
The equality follows in exactly the same fashion as Proposition \ref{compellp}.

An important subalgebra of the symbol algebra is the subalgebra of classical symbols. Assume that $\kappa$ is a continuous action of $\R_+$ on $\Ko(P)$ as strictly continuous bundle automorphisms. A symbol $a\in \Sigma^{m,p}(X,\Ko(P))$ is said to be homogeneous if $a(x,\lambda\xi)=\lambda^m\kappa(\lambda)( a(x,\xi))$ for large $\xi$ and $\lambda\geq 1$. If we for a very large $N$ can find $a_0, a_1,\ldots, a_{N-1}$, with $a_j\in \Sigma^{m-j,p}(X,\Ko(P))$ homogeneous, and $a-\sum a_j\in \Sigma^{m-N,p}(X,\Ko(P))$ we say that $a$ is classical. The subspace $\Sigma^{m,p}_{cl}(X,\Ko(P))\subseteq \Sigma^{m,p}(X,\Ko(P))$ of classical symbols forms a subalgebra. We have an inclusion $\Sigma^{m,p}_{cc}(X,\Ko(P))\subseteq\Sigma^{m,p}_{cl}(X,\Ko(P))$ since a compactly supported symbol is homogeneous of any order by definition. Again, the action $\kappa$ does not play an important role for the calculus and we will assume that $\kappa$ is trivial. 

The classical symbols form a much smaller class of symbols, but for the purpose of index theory they capture the entire picture. Later on we will consider elliptic operators, that is, symbols invertible outside a compact subset of $T^*X$. The following Proposition is proved in precisely the same manner as Lemma \ref{pathlem}.

\begin{prop}
\label{projpath}
If $a\in \Sigma^{0,p}(X,\Ko(P))$ is invertible outside a compact subset, there is a smooth path $(a_t)_{t\in[0,1]}\subseteq \Sigma^{0,p}(X,\Ko(P))+C^{N,\infty}_c(T^*X,\ellL^p(\pi^*P))$ of symbols invertible outside the same compact with $a^1\in \Sigma_{cl}^{m,p}(X,\Ko(P))+C^{N,\infty}_c(T^*X,\ellL^p(\pi^*P))$ and $a^0=a$.
\end{prop}

Motivated by Proposition \ref{homsymbprop}, we define
\[\Sigma^p(X,\Ko(P)):=\Sigma^{0,p}_{cl}(X,\Ko(P))/\left(\Sigma^{-1,p}_{cl}(X,\Ko(P))+\Sigma^{0,p}_{cc}(X,\Ko(\pi^*P))\right).\]
We can consider $\Sigma^p(X,\Ko(P))$ as a subalgebra of $C^\infty(S^*X,\Bo(\pi^*P))$ via the symbol mapping. We will define $\Sigma(X,\Ko(P))$ as the $C^*$-algebra closure of $\Sigma^p(X,\Ko(P))$. Let $\overline{\Sigma}(X,\Ko(P))$ denote the $C^*$-closure of $\Sigma^{0,p}_{cl}(X,\Ko(P))$ in $C_b(T^*X,\Bo(\pi^*P))$. 

\begin{prop}
\label{tothomotopy}
The $C^*$-algebra $\overline{\Sigma}(X,\Ko(P))$ is homotopic to $C(X,\Bo(P))$.
\end{prop}

\begin{proof}
Since $C^\infty(X,\Bo(P))\subseteq \Sigma^{0,p}(X,\Ko(P))$ there is an inclusion $C(X,\Bo(P))\subseteq \overline{\Sigma}(X,\Ko(P))$.  Define $\pi_0:\overline{\Sigma}(X,\Ko(P))\to C(X,\Bo(P))$, $a\mapsto (x\mapsto a(x,0))$. The $*$-homomorphism is clearly well-defined since it is the restriction of a $*$-homomorphism $C_b(T^*X,\Bo(\pi^*P))\to C(X,\Bo(P))$. The mapping $\pi_0$ is a right inverse to the inclusion $C(X,\Bo(P))\subseteq \overline{\Sigma}(X,\Ko(P))$ and the composition with this inclusion is homotopic to the identity on $\overline{\Sigma}(X,\Ko(P))$ via $\pi_t(a)(x,\xi):=a(x,t\xi)$.
\end{proof}

\begin{lem}
\label{morbo}
There is an isomorphism $C(X,\Bo(P))\otimes\Ko\cong C(X,\Ko(P))\otimes \Bo$.
\end{lem}

\begin{proof}
If we take choose a Cech cocycle $(c_{\alpha\beta})\in Z^1((U_\alpha),PU(\He))$ to represent $P$, we may identify 
\begin{align*}
C(X,\Bo(P))\otimes\Ko&= \{(a_\alpha)\in C(U_\alpha,\Bo\otimes \Ko):a_\alpha (c_{\alpha\beta}\otimes 1)=(c_{\alpha\beta}\otimes 1)a_\beta\}\quad\mbox{and}\\
C(X,\Ko(P))\otimes \Bo&= \{(b_\alpha)\in C(U_\alpha,\Bo\otimes \Ko):b_\alpha (1\otimes c_{\alpha\beta})=(1\otimes c_{\alpha\beta})b_\beta\}.
\end{align*}
The bundle $P\times_XP^{op}$ is trivializable, so there is a Cech cochain $(\tilde{c}_\alpha)\in C^0((U_\alpha),U(\He\otimes\He))$ such that $c_{\alpha\beta}\otimes (c_{\alpha\beta}^*)^{op}=\tilde{c}_\alpha^*\tilde{c}_\beta$, or equivalently; 
\[\tilde{c}_\alpha (c_{\alpha\beta}\otimes 1)=(1\otimes c_{\alpha\beta})\tilde{c}_\beta.\]
Therefore, if $(b_\alpha)\in C(X,\Ko(P))\otimes \Bo$ then $(\tilde{c}^*_\alpha b_\alpha\tilde{c}_\alpha)\in C(X,\Bo(P))\otimes\Ko$. This 
defines the sought after isomorphism.
\end{proof}

\begin{prop}
\label{strucsymb}
We have that $\Sigma(X,\Ko(P))=C(X,\Bo(P))+C(S^*X,\Ko(\pi^*P))$ and the $C^*$-algebra $\overline{\Sigma}(X,\Ko(P))$ fits into the short exact sequence
\[ 0\to C_0(T^*X,\Ko(\pi^*P))\to \overline{\Sigma}(X,\Ko(P))\to \Sigma(X,\Ko(P))\to 0.\]
\end{prop}

\begin{proof}
The inclusion $\Sigma(X,\Ko(P))\subseteq C(X,\Bo(P))+C(S^*X,\Ko(\pi^*P))$ is obvious since the variation of a symbol in $\Sigma^p(X,\Ko(P))$ in the fiber directions is compact. The reverse inclusion is also obvious since $C^\infty(X,\Bo(P))+C^\infty(S^*X,\ellL^1(\pi^*P))\subseteq \Sigma^{p}(X,\Ko(P))$. The symbol mapping $\sigma:\Sigma^{0,p}_{cl}(X,\Ko(P))\to C^\infty(S^*X,\Bo(\pi^*P))$, whose image is $\Sigma^p(X,\Ko(P))$, has kernel $\Sigma^{-1,p}_{cl}(X,\Ko(P))+\Sigma^{0,p}_{cc}(X,\Ko(P))$. It is clear from the estimate \eqref{thirdest} with $\alpha=0$, which is allowed for $m<0$, that 
\[\Sigma^{-1,p}_{cl}(X,\Ko(P))+C^{N,\infty}_c(T^*X,\ellL^p(P))\cap\Sigma^{0,p}_{cl}(X,\Ko(P))\subseteq C_0(T^*X,\Ko(\pi^*P)).\]
The embedding is clearly dense, since the left hand side contains the dense subalgebra $C^\infty_c(T^*X,\ellL^1(\pi^*P))$. The image of the extended symbol morphism $\sigma:\overline{\Sigma}(X,\Ko(P))\to \Sigma(X,\Ko(P))$ is dense, thus $\sigma$ is surjective. 
\end{proof}

A consequence of this Proposition is that the gluing construction \eqref{mpdef} of the symbols of order $(m,p)$ also can be carried out for the principal symbol $C^*$-algebra $\Sigma(X,\Ko(P))$ and $\overline{\Sigma}(X,\Ko(P))$. As in section \ref{sectionone} we have the result:

\begin{lem}
The dense embeddings $\Sigma^p(X,\Ko(P))\subseteq \Sigma(X,\Ko(P))$ and $\Sigma^{0,p}_{cl}(X,\Ko(P))\subseteq \overline{\Sigma}(X,\Ko(P))$ are isoradial.
\end{lem}

Now, let us turn to general principal $PU(\He)$-bundles. The $C^*$-algebra of symbols can after some minor modifications be defined also in this setting. We let $C_{r}(T^*X,\Ko(\pi^*P))$ denote the algebra of sections $T^*X\to \Ko(\pi^*P)$ that admits radial limits in norm topology. We define the $C^*$-algebras
\[\overline{\Sigma}_{co}(X,\Ko(P)):=C_{r}(T^*X,\Ko(\pi^*P))+\mathcal{M}(C(X,\Ko(P))),\]
\[\Sigma_{co}(X,\Ko(P)):=C(S^*X,\Ko(\pi^*P))+\mathcal{M}(C(X,\Ko(P))).\]
The same kind of results as above holds also for this pair of $C^*$-algebras but with $C(X,\Bo(P))$ replaced by the multiplier algebra $\mathcal{M}(C(X,\Ko(P)))$. In the case that $P$ is trivial, recall that $\mathcal{M}(C(X,\Bo))$ is the algebra of continuous functions to $\Bo$ equipped with its strong topology, see \cite{tommyboy}. There is a short exact sequence 
\begin{equation}
\label{coexactsymbol} 
0\to C_0(T^*X,\Ko(\pi^*P))\to \overline{\Sigma}_{co}(X,\Ko(P))\to \Sigma_{co}(X,\Ko(P))\to 0.
\end{equation}
and $\overline{\Sigma}_{co}(X,\Ko(P))$ is homotopic to $\mathcal{M}(C(X,\Ko(P)))$.

In the case that $P$ is defined from a set of mappings $c_{\alpha\beta}:U_{\alpha\beta}\to U(\He)$ that are smooth as mappings $c_{\alpha\beta}:U_{\alpha\beta}\to \Bo(\He)$ in the strong topology, we can also define smooth symbols in this setting. Smoothness of symbols defined in this way of course depend on the choice of the mappings $c_{\alpha\beta}$. In the examples we consider this notion is useful since such a choice comes naturally from local trivializations of the underlying geometries.

\subsection{Remark: Loop group Dirac operators}

One of the most studied principal $PU(\He)$-bundles, which has attracted much attention from the physicists, comes from projective representations of loop groups. If $G$ is a compact, connected, simply connected, simple Lie group the group $H^3(G,\Z)$ as well as $H^3_G(G,\Z)$, with $G$ acting by conjugation on itself, is free of rank one. The $PU(\He)$-bundle associated with the generator, known as the Wess-Zumino-Witten term by physicists, can be constructed from a projective representation of the loop group of $G$. If $\He$ is a projective representation of the loop group $\ellL G$ and $\mathcal{P}G$ denotes the path group, then $\mathcal{P}G\times_{\ellL G}\times PU(\He)\to G$ is a $G$-equivariant $PU(\He)$-bundle.

A celebrated result of Freed-Hopkins-Teleman \cite{fhtii} calculates the equivariant $K$-theory twisted by the Wess-Zumino-Witten term in terms of the Verlinde algebra of the Lie group. This construction does in a very straightforward factor over the equivariant $K$-theory of $\Sigma^{0,2+}(G,\Ko(P))$ since the loop group Dirac operator implementing the Freed-Hopkins-Teleman isomorphism squares to an operator that behaves as a first-order elliptic operator on the circle (see Lemma $2.45$ and $(3.36)$ of \cite{fhtii}). Verlinde algebras are finite so these equivariant twisted $K$-homology groups and twisted $K$-theory groups are always torsion groups. It follows from Theorem $3.4$ of \cite{mgoffpv} that also the non-equivariant twisted $K$-theory groups of these types of Lie groups are torsion groups. Thus, both the equivariant and the non-equivariant index pairing vanish in this case.

\subsection{Example: Constructing projective symbols on $T$-duals}
\label{tdualityonsymbols}

To give examples of how to construct projective symbols we turn to the $T$-dual bundles of section \ref{tdualbundle}. The construction of symbols on the $T$-duals is merely a usage of the extension of the isomorphism $\lambda$. Most calculations can be made explicitly on the $T$-dual projective bundles using ordinary index theory for pseudo-differential operators on circle bundles. This $PU(L^2(S^1))$-bundle is not smooth, since the cocycles are not norm continuous. We will however introduce an artificial notion of smoothness in this setting. 

Recall the setup from section \ref{tdualbundle}; we have a principal $S^1$-bundle $Z\to X$ and the $T$-dual principal $PU(L^2(S^1))$-bundle $P^T\to X\times S^1$ is constructed so that $C(Z)\ltimes \R\cong C(X\times S^1,\Ko(P^T))$. We will fix a finite cover $(U_\alpha)$ of $X$, trivializations $f_\alpha:\pi_Z^{-1}(U_\alpha)\to U_\alpha\times S^1$ and a Cech cochain $(\tilde{u}_{\alpha\beta})\in C^1((U_\alpha),\R)$ such that $f_\alpha f_\beta^{-1}(x,\theta)= (x,\theta+\tilde{u}_{\alpha\beta}(x))$. The collection of this data will be denoted by $\mathfrak{U}$. A section $a\in C(X\times S^1,\Ko(P^T))$ is said to be smooth if it comes from a collection $(a_\alpha)\in \oplus C^\infty(U_\alpha\times S^1,\Ko)$ and satisfies the cocycle condition with respect to the projective unitary-valued cocycle \eqref{tdualcocycle}. We will denote the Frechet algebra of smooth sections by $C^\infty_\mathfrak{U}(X\times S^1,\Ko(P^T))$, this is a dense isoradial subalgebra of $C(X\times S^1,\Ko(P^T))$.

Let us now construct symbols on $P^T$ by using symbols on $Z\times S^1$. We can define the symbol algebras $\Sigma^{m,p}_\mathfrak{U}(X\times S^1,\Ko(P^T))$ and $\Sigma^p_\mathfrak{U}(X\times S^1,\Ko(P^T))$ with respect to $\mathfrak{U}$. Observe that these are not norm dense in $\overline{\Sigma}_{co}(X\times S^1,\Ko(P^T))$ respectively $\Sigma_{co}(X\times S^1,\Ko(P^T))$, but nevertheless the embeddings induce isomorphisms on $K$-theory.

If $a$ is a pseudo-differential symbol of order $m$ on $Z\times S^1$, i.e. $a\in S^m(Z\times S^1)$, we can write $a=\sum_{l\in \Z} a_l \e^{2\pi il\theta}$ with $a_l\in C^\infty(T^*Z\times \R)$. We have local coordinates $(x,t,\theta,\xi,s,u)$ on $T^*(Z\times S^1)$, with $(t,s)$ being coordinates in the vertical direction of the circle fiber and its cotangent direction and similarly $(\theta,u)$ are coordinates on the trivial circle direction and its cotangent direction. We can write $a=a(x,t,\theta,\xi,s,u)$ and $a_l=a_l(x,t,\xi,s,u)$. Locally we define the function $\tilde{a}:= a(x,\theta,\xi,s,t,u)$. The locally defined $\Bo(\mathbbm{H})$-valued section $(x,\theta,\xi,s)\mapsto (\id\otimes \tilde{\lambda}_\theta)(\tilde{a})(x,\theta, \xi,s)$ can by the same calculation as in \eqref{tdualcocyclecondition} be shown to glue together to a global element $\mathfrak{q}_T(a)\in \mathcal{M}(C_0(T^*(X\times S^1),\Ko(P_Z)))$ and that the construction of $\mathfrak{q}_T(a)$ does not depend on the choice of coordinates. We can write this out more explicitly by letting $\mathfrak{q}_T(a)$ act on an $f\in \He_\theta$ as
\begin{align}
\label{qtconstruction}
\mathfrak{q}_T(a)(x,\theta, \xi,s)f(t)&:=\int_\R \Fg_6 a(x,t,\theta,\xi,s,t-u)f(u)\rd u=\\
\nonumber
&=\sum_{k,l}\int_0^1 \Fg_5 a_l(x,t,\xi,s,t+k-l-u)\e^{2\pi ik\theta} f(u)\rd u,
\end{align}
here $\Fg_6$ and $\Fg_5$ denotes Fourier transform in the sixth respectively fifth variable. The section $\mathfrak{q}_T(a)$ is constructed from gluing together local sections of the form $\id\otimes\tilde{\lambda}(\tilde{a})$ and while these local sections, by \eqref{thetatozero}, are unitarily equivalent to fiberwise defined operator-valued symbols we can conclude the following:

\begin{prop}
\label{qtconstruct}
If $Z\to X$ is a principal $S^1$-bundle with $T$-dual principal $PU(L^2(S^1))$-bundle $P_Z\to X\times S^1$, $\mathfrak{q}_T$ defines a linear mapping 
\[\mathfrak{q}_T:S^m(Z\times S^1)\to \Sigma^{m,1+}_\mathfrak{U}(X\times S^1,\Ko(P_Z)).\] 
\end{prop}

There is a very simple example of a $T$-dual, consider the $S^1$-principal bundle $S^3\to S^2$. Compare with Example $4.6$ of \cite{raeros}. Let us view $S^2$ as two unit discs $D_+, D_-\subseteq \C$ glued together along the boundary and $S^3$ as $D_+\times S^1$ glued together with $D_-\times S^1$ along the mapping $\partial D_+\times S^1\to \partial D_-\times S^1$, $(z,t)\mapsto (z,zt)$, where we view $S^1$ as the unit circle $|z|=1$ in $\C$. Therefore, the cocycle condition on $a_+\oplus a_- \in C_c(\overline{D}_+\times S^1\times \R)\oplus C_c(\overline{D}_-\times S^1\times \R)$ is that $a_+(z,t,s)=a_-(z,zt,s)$ for $|z|=1$. Observe that with this choice of cocycle, to glue together the field of Hilbert spaces $\He_\theta$ over $S^2\times S^1$ one would need sections that have monodromy $\e^{2\pi i\theta}$ when going one way around the equator in $S^2\times \{\theta\}$.

\subsection{Example: Symbols on the twisted group action of section \ref{twistedgroupbundle}}
\label{projequidiff}

Recall the principal $PU(L^2(M))$-bundle $P_{Z,\eta}\to X/\Gamma$ constructed above in section \ref{twistedgroupbundle}. We will speak of smooth sections in terms of smooth equivariant sections on $X$. So if we choose the data $\mathfrak{U}$ consisting of a cover and equivariant trivializations as in section \ref{twistedgroupbundle}  we can speak of smooth sections $C^\infty_\mathfrak{U}(X,\Ko(L^2(Z|X)))\subseteq C(X,\Ko(L^2(Z|X)))$ and we can also define $\Sigma_\mathfrak{U}^{m,p}(X,\Ko(L^2(Z|X)))$ as the elements of $C^\infty_\mathfrak{U}(X,\Ko(L^2(Z|X)))$ that satisfies the estimates \eqref{firstest}--\eqref{thirdest}. Since $\Gamma$ acts isometrically then $\Gamma$ acts on this algebra of symbols. The symbols of order $(m,p)$ over $P_{Z,\eta}\to X/\Gamma$ is defined as 
\[\Sigma^{m,p}_\mathfrak{U}(X/\Gamma, \Ko(P_{Z,\eta})):=\Sigma^{m,p}_\mathfrak{U}(X,\Ko(L^2(Z|X)))^\Gamma.\]

One natural way to obtain a $\Gamma$-invariant element of $\Sigma^{m,p}_\mathfrak{U}(X,\Ko(L^2(Z|X)))$ is to start with an elliptic  operator $\mathcal{D}$ on $Z$ that commutes with the projective $\Gamma$-action. That $\mathcal{D}$ commutes with the projective $\Gamma$-action means that $\e^{i\phi_\gamma}\gamma \mathcal{D}=\mathcal{D}\e^{i\phi_\gamma}\gamma$. Using functional calculus we can define $F_\mathcal{D}:=\mathcal{D}(1+\mathcal{D}^*\mathcal{D})^{-1/2}$ which is a zero order elliptic $\Gamma$-invariant pseudo-differential operator on $Z$. Since $M$ is compact there is an operator-valued symbol on $X$ defined from $F_\mathcal{D}$ which is a $\Gamma$-invariant element of $\Sigma^{0,p}_{\mathfrak{U}}(X,\Ko(L^2(Z|X)))$ for any $p>\dim(M)$.

\subsection{$K$-theory of the projective symbol algebra}

If $P\to X$ is a principal $PU(\He)$-bundle with Dixmier-Douady invariant $\omega$ the $K$-theory $K_*(C_0(X,\Ko(P)))$ is an invariant of $P$ and in fact $K_*(C_0(X,\Ko(P)))$ is determined up to isomorphism by $X$ and $\omega$. We will denote a choice of isomorphism class by $K^*(X,\omega)$. The isomorphism between different choices of representatives for $K^*(X,\omega)$ is unfortunately not canonical. Twisted $K$-theory is a rich invariant carrying many similarities with ordinary topological $K$-theory. For instance, if $P,P'$ are two principal $PU(\He)$-bundles the isomorphism $\Ko(P)\otimes_X\Ko(P')\cong \Ko(P\times_X P')$ defines a cup product 
\[K^*(X,\delta(P))\times K^*(X,\delta(P'))\to K^*(X,\delta(P)+\delta(P')).\]
The cup product makes every twisted $K$-theory group over $X$ a module over $K^*(X)$.

If $a\in \Sigma^{0,p}(X, \Ko(P))$ is an elliptic projective symbol we can define a symbol class $\Xi[a]\in K^0(T^*X,\pi^*\delta(P))\cong KK(\C,C_0(T^*X,\Ko(P)))$, the construction can word by word be generalized to any elliptic $a\in \overline{\Sigma}_{co}(X,\Ko(P))$. The abstract approach is to use Proposition \ref{projpath} and Proposition \ref{strucsymb}. With $a$ we can associate a classical symbol $a_0$ by Proposition \ref{projpath} and a class $[\sigma(a^0)]\in K_1(\Sigma^p(X,\Ko(P)))$. Then we may let $\Xi[a]$ be the image of $[\sigma(a^0)]$ under the boundary mapping coming from the short exact sequence of Proposition \ref{strucsymb}. 

The direct construction of the symbol class $\Xi[a]$ is to take a parametrix $r\in \Sigma^{0,p}(X,\Ko(P))$ to $a$. That is, $ar-1,ra-1\in C^\infty_0(T^*X,\Ko(\pi^*P))$. The $KK$-class $\Xi[a]\in KK(\C,C_0(T^*X,\Ko(\pi^*P)))$ can be represented by the even $\C-C_0(T^*X,\Ko(\pi^*P))$-Kasparov module $(C_0(T^*X,\Ko(\pi^*P)\otimes \C^2),F_a)$ where $\C^2$ is graded via $1\oplus -1$ and $F_a\in \mathcal{M}(C_0(T^*X,\Ko(\pi^*P)\otimes \C^2))$ is defined in terms of the matrix
\[F_a:=\begin{pmatrix} 0& a\\ r& 0\end{pmatrix}.\]
Analogously, we can define the "odd" symbol mapping $\Xi:K_0(\Sigma(X,\Ko(P)))\to K^1(T^*X,\pi^*\delta(P))$ either in terms of Proposition \ref{strucsymb} or in terms of Kasparov modules.

\begin{sats}
\label{gendd}
The symbol class mapping $\Xi:K_*(\Sigma(X,\Ko(P)))\to K^{*-1}(T^*X,\pi^*\delta(P))$ associated with the short exact sequence of Proposition \ref{strucsymb} is an isomorphism. The same statement holds for the compact-open symbol class mapping $K_*(\Sigma_{co}(X,\Ko(P)))\to K^{*-1}(T^*X,\pi^*\delta(P))$ associated with the short exact sequence \eqref{coexactsymbol}.
\end{sats}

\begin{proof}
If we take the six term exact sequence of the short exact sequence of Proposition \ref{strucsymb}, the first statement of the theorem follows if $K_*(\overline{\Sigma}(X,\Ko(P)))=0$. By Lemma \ref{tothomotopy}, $K_*(\overline{\Sigma}(X,\Ko(P)))=K_*(C(X,\Bo(P))$. But $C(X,\Bo(P))$ is Morita equivalent to $C(X,\Ko(P))\otimes \Bo$ by Lemma \ref{morbo} so the K\"unneth theorem implies that $K_*(C(X,\Bo(P))=0$. The second statement of the theorem follows in a similar manner since $K_*(\overline{\Sigma}_{co}(X,\Ko(P)))=K_*(\mathcal{M}(C(X,\Ko(P))))=0$.
\end{proof}

\subsection{Twisted Chern characters of projective symbols}

We will now turn to the homological picture that will give us Atiyah-Singer type formulas for the twisted index pairing. Let us recall the definition of twisted de Rham cohomology of a smooth manifold $X$. Assume that $\Omega$ is a closed $3$-form on $X$ and define the differential operator $\rd_\Omega:=\rd-\Omega$ satisfying $\rd_\Omega^2=0$. The twisted de Rham cohomology $H^*_\Omega(X)$ is defined as the $\Z/2\Z$-graded cohomology of the $\Z/2\Z$-graded complex of differential forms on $X$ equipped with $\rd_\Omega$. The vector space $H^*_\Omega(X)$ is finite-dimensional which is easily seen from Hodge theory and the isomorphism class of $H^*_\Omega(X)$ is determined by the de Rham class of $\Omega$. 

The twisted Chern character is a natural mapping $\ch_P:K^*(X,\delta(P))\to H^*_\Omega(X)$, where $\Omega$ is a certain $3$-form representing the image of $\delta(P)$ in de Rham cohomology. The twisted Chern character can be constructed either using the spectral sequence of Atiyah-Segal, see \cite{atiyahsegalcohom}, or a generalization of Connes-Hochschild-Kostant-Rosenberg theorem due to Mathai-Stevenson, see \cite{mathstev}. In the compact-open topology one must settle for the spectral sequence construction. In the smooth setting, the twisted Chern character is defined as the periodic Connes-Chern character $\ch_{cc}:K_*(C^\infty(X,\ellL^1(P)))\to HP_*(C^\infty(X,\ellL^1(P)))$ composed with a natural transformation $HP_*(C^\infty(X,\ellL^1(P)))\to H^*_\Omega(X)$. We refer the reader to the construction in \cite{mathstev}. The twisted Chern character behaves much like the usual Chern character.

\begin{lem}
\label{cherncalc}
The Chern character $\ch_\omega:K^*(X,\omega)\to H^*_\Omega(X)$ is multiplicative in the sense that for any $[a]\in K^*(X,\omega)$ and $[E]\in K^*(X,\omega')$ we have that
\[\ch_{\omega+\omega'}([a]\cup [E])=\ch_\omega[a]\wedge \ch_{\omega'}[E]\in H^*_{\Omega+\Omega'}(X).\]
\end{lem}

The Lemma is a direct consequence of the construction in \cite{mathstev} and Corollary $2.8$ of \cite{gorokh}, the latter result stating that the Chern character on generalized cycles is multiplicative.

On an abstract level, we can define the Chern character on the projective principal symbol algebra $\widetilde{\ch_P}:K_*(\Sigma(X,\Ko(P)))\to H_{c,\pi^*\Omega}^{*}(T^*X)$ as the isomorphism of Lemma \ref{gendd} composed with the Chern character $\ch_{\pi^*P}$, using either of the two twisted Chern character depending on whether $P$ is smooth or not. However, the nature of the projective symbols allow us to construct it explicitly when $P$ is smooth using the generalized Connes-Hochschild-Kostant-Rosenberg theorem and the regularizing techniques of \cite{grigone} and \cite{grigtwo}. This of course depend on the following lemma.

\begin{lem}
\label{twistedschattenrep}
For any invertible $a\in \Sigma^{p}(X,\Ko(P))$ there is an invertible $a_0\in C^\infty(X,\Bo(P\oplus P))$ such that 
\[(a\oplus 1)-\pi^*a_0\in C^{N,\infty}(S^*X,\ellL^p(P\oplus P))\cap \Sigma^p(X,\Ko(P\oplus P)).\]
\end{lem}

\Large
\section{The index pairing with twisted $K$-theory}
\label{pair}
\normalsize

We will in this section construct twists of pseudo-differential operators with values over a principal $PU(\He)$-bundle by projective bundles over the opposite $PU(\He)$-bundle, i.e. projective modules over the corresponding continuous trace algebra. The index theory of these operators defines an index pairing between the $K$-theory of the projective principal symbol algebra and the opposite twisted $K$-theory. This index pairing coincides with the twisted index pairing and using Theorem \ref{trivddindex} the index can be calculated by means of an Atiyah-Singer type index theorem. The main idea of this section can be read out from the following commuting diagram where all arrows are isomorphisms:
\begin{equation}
\label{bigpicture}
\xymatrix@C=4.4em@R=5.71em{
& K_*^{geo}(X,P)\ar[ddr]^{\mu}\ar[ddl]_{\mathfrak{F}} &  \\
& K^*(T^*X,\pi^*\delta(P))\ar[dr]^{PD}\ar[u]_{c}&\\
K_{*-1}(\Sigma(X,\Ko(P)))\ar[rr]^{Op_\Sigma}\ar[ur]_{\Xi}   && KK_*(C(X,\Ko(P^{op})),\C))} 
\end{equation}
In \cite{baumdtva} the Atiyah-Singer index theorem was attributed to the upper right part of this commuting diagram. The mapping $\mu$ denotes the assembly mapping of \cite{carwang}. $PD$ is Poincar\'e duality that maps an element to its external product with the spin$^c$-Dirac operator on $T^*X$, see more in \cite{tutwisted}. In even degree, when $P$ is trivializable, $PD$ is implemented by choosing an operator from a difference class on $T^*X$. The "choose an operator"-mapping $Op_\Sigma$ is just as in the classical setting the construction of an operator from a symbol. We construct this mapping in the subsequent section and we prove in Theorem \ref{chooseanop} that $Op_\Sigma$ makes the lower triangle in the diagram commute.

We will only make use of the lower triangle in this paper. The full diagram is similar to the solution of the index problem for hypo-elliptic operators in the Heisenberg calculus, see the diagram in Section $1.3$ of \cite{baumvanerp}. Constructing the clutching mapping $c$ and the de-clutching mapping $\mathfrak{F}$, i.e. representing our classes in geometric $K$-homology, would produce more natural solutions to the index problem. The clutching mapping $c$, in the case of trivializable $P$, associates a geometric cycle with an elliptic complex over $T^*X$ and solves The Index Problem, as defined in \cite{baumdtva}, for elliptic pseudo-differential operators. For trivializable $P$ the de-clutching mapping $\mathfrak{F}$ is more or less taking the fiberwise Dirac operator. More precisely, if $(M,E,f)$ is a geometric cycle over $X$ and $f$ is smooth, we factor $f$ over an embedding into $X\times S^N$, for a large $N$, and the trivial fiber bundle $X\times S^N\to X$. A representative of $\mathfrak{F}(M,E,f)$ is given as the operator-valued symbol constructed in terms of fiberwise defined operator-valued symbol on $X$ of the operator on $X\times S^N\to X$ given by the external product of the spin$^c$-Dirac operator on $M$ twisted by $E$ and the fiberwise spin$^c$-Dirac operator on the normal bundle of $M$ in $X\times S^N$.  In general, $c$ and $\mathfrak{F}$ are very hard to construct. An explicit construction of $\mathfrak{F}$ and its inverse seems vital to understand the index theory of projective operators.

\subsection{Choose an operator construction}
\label{chooseoperator}

To associate $K$-homology classes with elliptic projective symbols we must construct operators from symbols. We will construct this as a linear mapping $\tau:\Sigma^{m,p}(X,\Ko(P))\to \Bo(L^2(X,\He\otimes\He^*))$ in the norm smooth setting. We will show existence of such a linear mapping in the general case, but the construction is in practice quite complicated and uses complex analytic methods.  

The main idea in the operator construction is to represent our symbols on the Hilbert-Schmidt sections of the Azumaya bundle. If $P\to X$ is a principal $PU(\He)$-bundle there is an isomorphism $\phi:C(X,\Ko(P^{op})\otimes_X\Ko(P))\cong C(X,\Ko\otimes\Ko)$. Since this isomorphism is $C(X)$-linear we obtain the following Proposition.

\begin{prop}
\label{tensiso}
The isomorphism $\phi$ extends to isomorphisms
\[\tilde{\phi}:C(X,\Ko(P^{op}))\otimes_{C(X)} \Sigma_{co}(X,\Ko(P))\to \Sigma_{co}(X,\Ko)\otimes\Ko\quad\mbox{and}\]
\[\tilde{\phi}:C(X,\Ko(P^{op}))\otimes_{C(X)} \Sigma(X,\Ko(P))\to \Sigma(X,\Ko)\otimes\Ko\quad\mbox{if $P$ is smooth}.\]
If $P$ is smooth we can take this isomorphism to map tensor products of smooth sections to smooth functions.
\end{prop}

Let $L^2(X,\ellL^2(P))$ denote the $L^2$-sections of the Hilbert-Schmidt class operators over $P$. There is a representation $\pi:C(X,\Ko(P^{op}))\otimes_{C(X)}C(X,\Ko(P))\to \Bo(L^2(X,\ellL^2(P)))$ by letting $C(X,\Ko(P^{op}))$ act by multiplication on the right and letting $C(X,\Ko(P))$ act by multiplication on the left. As a Hilbert space bundle $\ellL^2(P)\cong X\times \He\otimes \He^*$ and this isomorphism intertwines $\pi$ with the pointwise action on $L^2(X,\He\otimes \He^*)$ defined from the isomorphism $\phi$. 

We can now define $\tau$ for a smooth $P$. If $a\in \Sigma^{m,p}(X,\Ko(P))$ the function $\tilde{\phi}(1\otimes a)\in C^\infty(T^*X,\Bo(\He\otimes\He^*))$ satisfies the estimates 
\[\|\partial_x^\alpha\partial_\xi^\beta \tilde{\phi}(1\otimes a)(x,\xi)\|_{\Bo(\He\otimes\He^*)}\lesssim (1+|\xi|)^{m-|\beta|},\]
in local coordinates. Thus it makes sense to define
\[\tau(a):=\tilde{\phi}(1\otimes a)(x,\partial).\]

For general $P$ the construction of $\tau$ does not work. One very direct reason for this is that $\tau$ is not bounded in the norm topology. Since $X$ is smooth, we can choose a real analytic structure on $X$ and for $\epsilon>0$ we can construct a Grauert tube $X_\epsilon$ around $X$, see more in \cite{guille}. The Grauert tube $X_\epsilon$ is a strictly pseudo-convex domain in a Stein manifold that contains $X$ as a totally real submanifold and is diffemorphic to $T^*X$. We will let $\pi^\epsilon:X_\epsilon\to X$ denote the projection and by $H^2(\partial X_\epsilon)$ we denote the Hardy space on the boundary of $X_\epsilon$. By \cite{epmel}, integration along the fiber $\pi_*^\epsilon:H^2(\partial X_\epsilon)\to L^2(X)$ is an isomorphism for $\epsilon$ small enough. Furthermore, by Theorem $5.2$ of \cite{guille}, $\pi_*^\epsilon$ intertwines pseudo-differential operators on $X$ and Toeplitz operators on $\partial X_\epsilon$ up to lower order operators. That is, if $P_\epsilon:L^2(\partial X_\epsilon)\to H^2(\partial X_\epsilon)$ denotes the Szeg\"o projection then for any pseudo-differential symbol $a$ we have that $\pi_*^\epsilon P_\epsilon aP_\epsilon(\pi_*^\epsilon)^{-1} -a(x,\partial)$ is a lower order operator. Motivated by this we define
\[\tau_{co}:\Sigma_{co}(X,\Ko(P))\to \Bo(L^2(X,\He\otimes \He^*)), \quad a\mapsto \pi_*^\epsilon P_\epsilon \left(\tilde{\phi}(1\otimes a)\right)P_\epsilon(\pi_*^\epsilon)^{-1}.\]

\begin{lem}
\label{tauproperties}
The mappings $\tau$ and $\tau_{co}$ almost commutes with $C(X,\Ko(P^{op}))$ in the sense that
\begin{itemize}
\item[i)] The commutator $[\tau_{co}(a),\tilde{\phi}(b\otimes 1)]\in \Ko(L^2(X,\He\otimes\He^*))$ for any $a\in \Sigma_{co}(X,\Ko(P))$, $b\in C(X,\Ko(P^{op}))$.
\item[ii)] If $P$ is smooth, $[\tau(a),\tilde{\phi}(b\otimes 1)]\in \Psi^{m-1,p}(X,\Ko(\He\otimes\He^*))$ for any $a\in \Sigma^{m,p}(X,\Ko(P))$, $b\in C^{N,\infty} (X,\ellL^1(P^{op}))$.
\end{itemize}
The mappings are almost multiplicative relative to $C(X,\Ko(P^{op}))$ in the sense that
\begin{itemize}
\item[iii)]  $(\tau_{co}(aa')-\tau_{co}(a)\tau_{co}(a'))\tilde{\phi}(b\otimes 1)\in \Ko(L^2(X,\He\otimes\He^*))$ for any $a,a'\in \Sigma_{co}(X,\Ko(P))$, $b\in C(X,\Ko(P^{op}))$.
\item[iv)] If $P$ is smooth, $(\tau(aa')-\tau(a)\tau(a'))\tilde{\phi}(b\otimes 1)\in\Psi^{m+m'-1,p}(X,\Ko(\He\otimes\He^*))$ for any $a\in \Sigma^{m,p}(X,\Ko(P))$, $a'\in \Sigma^{m',p}(X,\Ko(P))$ and $b\in C^{N,\infty} (X,\ellL^1(P^{op}))$.
\end{itemize}
Products with elements of $C(X,\Ko(P^{op}))$ satisfies that
\begin{itemize}
\item[v)] If $P$ is smooth, $\tau(a)\tilde{\phi}(b\otimes 1)\in \Psi^{m,p}(X,\Ko(\He\otimes\He^*))$ for any $a\in \Sigma^{m,p}(X,\Ko(P))$, $b\in C^{N,\infty} (X,\ellL^1(P^{op}))$.
\end{itemize}
The two quantizations $\tau$ and $\tau_{co}$ are equivalent relative to $C(X,\Ko(P^{op}))$ if $P$ is smooth in the sense that 
\begin{itemize}
\item[vi)] $(\tau_{co}(a)-\tau(a))\tilde{\phi}(b\otimes 1)\in \Ko(L^2(X,\He\otimes\He^*))$ for any $a\in \Sigma(X,\Ko(P))$, $b\in C(X,\Ko(P^{op}))$.
\end{itemize}
\end{lem}

We omit the proof of the Lemma since it follows directly from restricting to the local situation where one may use the calculus of section \ref{sectionone}. Using the linear mapping $\tau_{co}$, or $\tau$ whenever it is available, we can define the "choose an operator"-mapping $Op_\Sigma:K_{*-1}(\Sigma_{co}(X,\Ko(P)))\to KK_*(C(X,\Ko(P^{op})),\C)$.  We will denote the Hilbert space $L^2(X,\He\otimes\He^*)\otimes \C^2$ by $\He_X$. This is a graded Hilbert space with the grading operator 
\[\gamma_{\He_X}:=\begin{pmatrix} \id_{L^2(X,\He\otimes \He^*)}& 0\\ 0& -\id_{L^2(X,\He\otimes \He^*)}\end{pmatrix}.\] 
We define the representation 
\[\tilde{\pi}:C(X,\Ko(P^{op}))\to \Bo(\He_X)\quad\mbox{by}\quad\tilde{\pi}(b):=\tilde{\phi}(b\otimes 1)\oplus \tilde{\phi}(b\otimes 1).\] 
If $a\in \Sigma_{co}(X,\Ko(P))$ is elliptic with parametrix $r$ we define the odd operator
\[F_{a,r}:=\begin{pmatrix} 0& \tau_{co}(a)\\ \tau_{co}(r)& 0\end{pmatrix}\in \Bo(\He_X).\]
It follows from Lemma \ref{tauproperties} that 
\[[F_{a,r}, \tilde{\pi}(b)]\in \Ko(\He_X)\quad\mbox{and}\quad \tilde{\pi}(b)(F_{a,r}^2-1)\in \Ko(\He_X)\quad\mbox{for}\quad b\in C(X,\Ko(P^{op})).\] 
Whenever $a$ is unitary, $\tilde{\pi}(b)(F_{a,r}-F_{a,r}^*)\in \Ko(\He_X)$ for $b\in C(X,\Ko(P^{op}))$. If $x\in K_1(\Sigma_{co}(X,\Ko(P)))$ is represented by a unitary $a$ with parametrix $r$ we define the $K$-homology class $Op_\Sigma(x)\in KK_0(C(X,\Ko(P^{op})),\C)$ as the even Fredholm module $(\tilde{\pi},\He_X,F_{a,r})$. By Lemma \ref{tauproperties}.v the analytic $K$-homology class $Op_\Sigma(x)$ does not depend on the choice of representative for $x$. 

When $P$ is smooth we can represent $x$ by an elliptic almost unitary $a\in \Sigma^{0,p}(X,\Ko(P))$ and by Lemma \ref{tauproperties}.vi the Kasparov operator $F_{a,r}$ is operator-homotopic to the odd operator 
\[\tilde{F}_{a,r}:=\begin{pmatrix} 0& \tau(a)\\ \tau(r)& 0\end{pmatrix}\in \Bo(\He_X).\]
The Fredholm module $(\tilde{\pi},\He_X,\tilde{F}_{a,r})$ is finitely summable over $C^{N,\infty}(X,\ellL^1(P))$ by Lemma \ref{tauproperties}.iv.

The even part $c:K_0(\Sigma_{co}(X,\Ko(P)))\to KK_1(C(X,\Ko(P^{op})))$ can be defined either by suspension or as follows; represent a class $x\in K_0(\Sigma_{co}(X,\Ko(P)))$ by $p\in \overline{\Sigma}_{co}(X,\Ko(P))$ such that $p^2-p,p^*-p\in C_0(T^*X,\Ko(P))$ and define $Op_\Sigma(x)$ as the odd Fredholm module $(\tilde{\phi},L^2(X,\He\otimes\He^*),F_p)$ where 
\[F_p:=2\tau_{co}(p)-1.\]

\begin{sats}
\label{chooseanop}
The mapping $Op_\Sigma:K_*(\Sigma_{co}(X,\Ko(P))\to KK_{*-1}(C(X,\Ko(P^{op})),\C)$ is an isomorphism making the lower triangle in the diagram \eqref{bigpicture} commute.
\end{sats}

\begin{proof}
Following \cite{tutwisted}, the mapping $PD$ is defined as the composition of
\begin{align*}
K^*(T^*X,\pi^*\delta(P))&=KK_*(\C,C_0(T^*X,\Ko(\pi^*P)))=KK^X_*(C(X),C_0(T^*X,\Ko(\pi^*P)))\to\\
&\to KK^X_*(C(X,\Ko(P^{op})),C_0(T^*X,\Ko(\pi^*P)\otimes_X \Ko(P^{op})))\xrightarrow{\phi^*}\\
&\xrightarrow{\phi^*} KK^X_*(C(X,\Ko(P^{op})),C_0(T^*X))\to KK_*(C(X,\Ko(P^{op})),\C),
\end{align*} 
where the last arrow is pairing with the Dirac operator $[D]$ on $T^*X$. The theorem follows from that for an elliptic $a\in \overline{\Sigma}_{co}(X,\Ko(P))$, $PD\circ \Xi[a]$ is represented by the Kasparov product of the $C(X)$-linear $C(X,\Ko(P^{op}))-C_0(T^*X,\Ko\otimes\Ko)$-Kasparov module $(\phi, C_0(T^*X,\Ko\otimes \Ko \otimes \C^2),\phi(F_a))$ and $[D]$. By Kasparov's index theorem, see Theorem $5$ of \cite{kaspindex}, this product is given by $Op_\Sigma[a]$ under Morita equivalence. 
\end{proof}

\subsection{The mapping dual to $Op_\Sigma$ and pairing with $K$-theory}
In this section we will study dilations of operators of the form $\tau(a)$ along a finitely generated projective right $C^\infty(X,\Ko(P^{op}))$-module $\Eg$. This construction defines a mapping $e_\Sigma:K_*(C(X,\Ko(P^{op})))\to K^{*-1}(\Sigma_{co}(X,\Ko(P)))$ which in a sense is dual to $Op_\Sigma$. In this section we will assume that $P$ is smooth, the general case works in the exact same way replacing $\tau$ with $\tau_{co}$ and modules over $C(X,\Ko(P^{op}))$. These indices give a concrete description of the twisted index pairing.  Since $\Eg$ is a finitely generated projective right $C^\infty(X,\Ko(P^{op}))$-module and $C^\infty(X,\Ko(P^{op}))$ is stable there is a projection $p_\Eg\in C^\infty(X,\Ko(P^{op}))$ such that $\Eg=p_\Eg C^\infty(X,\Ko(P^{op}))$. It even holds that $p_\Eg\in C^\infty(X,\ellL^1(P^{op}))$. We will assume that $p_\Eg$ is hermitian since we can represent any $K$-theory class by a hermitian projection and set $[\Eg]:=[p_\Eg]\in K_0(C^\infty(X,\Ko(P^{op}))$. Since any class in $K_0(C^\infty(X,\Ko(P^{op}))$ is a formal difference of such classes the construction of the general index pairing follows from this special case.

We define the Hilbert space $L^2(X,\Eg):=\tilde{\phi}(p_\Eg\otimes 1)L^2(X,\He\otimes \He^*)$ which is a closed subspace of $L^2(X,\He\otimes\He^*)$. Define the linear mapping $\tau_\Eg:\Sigma^{m,p}(X,\Ko(P))\to \Bo(L^2(X,\Eg))$ by
\[\tau_\Eg(a)= \tilde{\phi}(p_\Eg\otimes 1)\tau(a)|_{L^2(X,\Eg)}.\]
It follows from Lemma \ref{tauproperties} that $\tau_\Eg$ is multiplicative up to lower order terms, i.e. compact terms. Let $\Psi^{p}(X,\Eg)$ denote the linear span of $\tau_\Eg(\Sigma^{0,p}(X,\Ko(P)))$ and $\Ko_\Ko(L^2(X,\Eg))$, which by Lemma \ref{tauproperties} is a $*$-algebra. Let $\Psi(X,\Eg)$ be the $C^*$-algebra closure of $\Psi^p(X,\Eg)$, which clearly does not depend on $p$. By Lemma \ref{tauproperties}.vi we have the equality $\Psi(X,\Eg)=\tau_{co,\Eg}(\Sigma(X,\Ko(P)))+\Ko(L^2(X,\Eg))$. We can also define the classical operators $\Psi^{p}_{cl}(X,\Eg)$ and its $C^*$-closure  $\Psi_{cl}(X,\Eg)$. It follows from Proposition \ref{projpath} that any elliptic element in $\Psi(X,\Eg)$ is homotopic, through a path of elliptic elements, to an elliptic in $\Psi_{cl}(X,\Eg)$.

Define the principal symbol mapping $\sigma_\Eg:\Psi_{cl}(X,\Eg)\to \Sigma(X,\Ko(P))$ by $\tau_\Eg(a)\mapsto a_0$ where $a_0$ comes from an asymptotic expansion $a\sim \sum a_{-j}$ in homogeneous terms. By Lemma \ref{tauproperties} we have the short exact sequence 
\begin{equation}
\label{eeext}
0\to \Ko_\Ko(L^2(X,\Eg))\to \Psi_{cl}(X,\Eg)\xrightarrow{\sigma_\Eg} \Sigma(X,\Ko(P))\to 0.
\end{equation}
Using the completely positive mapping $\tau_{co}$ it follows that $\sigma_\Eg$ admits a completely positive splitting. A consequence of this is the following lemma.

\begin{lem}
\label{extlem}
The extension class $[\Psi_\Eg]\in Ext(\Sigma(X,\Ko(P)))$, defined from \eqref{eeext}, is in the image of the natural mapping $K^1(\Sigma(X,\Ko(P)))\to Ext(\Sigma(X,\Ko(P)))$.
\end{lem}

The association $\Eg\mapsto [\Psi_\Eg]$ defines a group homomorphism 
\[e_\Sigma:K_*(C_0(X,\Ko(P^{op})))\to K^{*+1}(\Sigma(X,\Ko(P)))\] 
by Lemma \ref{extlem}. In fact, the mapping $[\Eg]\mapsto [\Psi_\Eg]$ is given by the external product with a $KK$-element $\Psi$. The external product with $\Psi$ can be calculated although the principal symbol algebra is non-separable due to Lemma \ref{extlem}. We will use the suggestive notation 
\[\Delta^*:\Sigma(X,\Ko(P))\otimes C(X,\Ko(P^{op}))\to \Sigma(X,\Ko(P))\otimes_{C(X)} C(X,\Ko(P^{op}))\] 
for the quotient mapping. While $K$-homology is Morita invariant we can define the $K$-homology class $\Psi$ as the element
\[\Psi:=[\Psi^X\otimes \id_\Ko]\circ\tilde{\phi}\circ \Delta^*\in K^1(\Sigma(X,\Ko(P))\otimes C(X,\Ko(P^{op}))),\]
where $\tilde{\phi}$ is the isomorphism of Proposition \ref{tensiso} and $\Psi^X\in K^1(\Sigma(X,\Ko))$ is defined from the extension \eqref{psix}, which is a well defined $K$-homology class due to the same reasoning as in the proof of Lemma \ref{extlem}. 

\begin{sats}
\label{dualaavbildningen}
The short exact sequence of Proposition \ref{strucsymb} induces a well defined mapping $\Xi^*:K_*(T^*X,\pi^*\delta(P))\to  KK_{*-1}(\Sigma(X,\Ko(P)),\C)$ that fits into the commutative diagram:
\[\xymatrix@C=4.4em@R=4.71em{
& K_*(T^*X,\pi^*\delta(P))\ar[dr]^{\Xi^*}&\\
K_{*}(C(X,\Ko(P^{op})))\ar[rr]^{e_\Sigma}\ar[ur]^{PD^*}   && KK_{*-1}(\Sigma(X,\Ko(P)),\C))}.\]
\end{sats}

\begin{proof}
The short exact sequence of Proposition \ref{strucsymb} induces a well defined mapping $\Xi^*:K_0(T^*X,\pi^*\delta(P))\to  Ext(\Sigma(X,\Ko(P)))$, if we can show that this mapping satisfies that $\Xi^*\circ PD^*[\Eg]=[\Psi_\Eg]$ the Theorem follows from Lemma \ref{extlem}. Again, this equality follows from Kasparov's index theorem which implies that $\Xi^*\circ [D]=\Psi^X$.
\end{proof}

Recall the short exact sequence of Proposition \ref{strucsymb}. Since $\overline{\Sigma}(X,\Ko(P))$ is more or less homotopic to a stable multiplier algebra the symbol algebra $\Sigma(X,\Ko(P))$ has much in common with $C_0(T^*X,\Ko(\pi^*P))$. The latter $C^*$-algebra is a Poincar\'e dual to $C(X,\Ko(P^{op}))$, see more in \cite{tutwisted}, and the results of Theorem \ref{chooseanop} and Theorem \ref{dualaavbildningen} can be interpreted as coming from this duality. However, since $\Sigma(X,\Ko(P))$ is not separable it is not a Poincar\'e dual in the usual sense since Kasparov products for non-separable $C^*$-algebras can not be constructed in the usual way.  

\begin{sats}
\label{indexpairing}
Let $P\to X$ be a principal $PU(\He)$-bundle over the compact smooth manifold $X$. If $\Eg$ is a finitely generated projective $C^\infty(X,\Ko(P^{op}))$-module and $a$ is an elliptic projective symbol over $P$, the index of $\tau_\Eg(a)$ is given by
\begin{align*}
\ind(\tau_\Eg(a))&=([a]\otimes [\Eg]) \circ \Psi =\\
&=[\Eg]\circ Op_\Sigma[a]=[a]\circ e_\Sigma [\Eg]=\\
&\qquad\qquad\qquad\;\;\;=\ind_X(\Xi[a]\cup [\Eg]).
\end{align*}
In particular, if $P$ is smooth
\[\ind(\tau_\Eg(a))=\int _{T^*X} \widetilde{\ch_P}[a] \wedge \ch_{P^{op}}[\Eg]\wedge Td(X).\]
\end{sats}

\begin{proof}
The index formulas in $K$-theory follows from the definition of the index pairing, Theorem \ref{chooseanop} and Theorem \ref{dualaavbildningen}. We have that $\Delta^*([a]\otimes [\Eg])$ is the cup product of $[a]$ and $[\Eg]$. The index formula in de Rham cohomology now follows from Theorem \ref{trivddindex}, Lemma \ref{cherncalc} and the fact that $\Xi$ by naturality commutes with cup products of elements of twisted $K$-groups since these results imply that 
\begin{align*}
\ind(\tau_\Eg(a))&=\Psi^X\circ ([a]\cup [\Eg])=\int _{T^*X} \widetilde{\ch}\left(\widetilde{\phi}^*( [a]\cup [\Eg])\right) \wedge Td(X)=\\
&=\int _{T^*X}\ch\circ \Xi ( [a]\cup [\Eg]) \wedge Td(X)=\int _{T^*X}\ch (\Xi  [a]\cup [\Eg]) \wedge Td(X)=\\
&=\int _{T^*X}\ch_{\pi^*P} \Xi  [a]\wedge \ch_{P^{op}} [\Eg] \wedge Td(X)=\int _{T^*X}\widetilde{\ch_P}[a]\wedge \ch_{P^{op}}[\Eg]\wedge Td(X).
\end{align*}
\end{proof}

Observe that we in this section rarely use that $\Eg$ is a $C(X,\Ko(P^{op}))$-module, only that $\Eg$ is a $C(X,\Ko(P'))$-module for a principal $PU(\He)$-bundle $P'$ such that $P\times_X P'$ is trivializable.

\begin{cor}
The index pairing between projective pseudo-differential operators with values over $P$ and finitely generated projective modules over $P^{op}$ expresses the index pairing 
\[K^*(T^*X,\pi^*\omega)\times K^*(X,-\omega)\xrightarrow{\cup} K^*(T^*X)\xrightarrow{\ind_X} \Z,\]
that is, the following diagram commutes:
\[
\small
\xymatrix@C=2.9em@R=5.71em{
K^*(T^*X,\pi^*\omega)\times K^*(X,-\omega)\ar[r] \ar[dr]_{\ch_{\pi^*\omega}\times \ch_{-\omega}}& K_*(\Sigma(X,\Ko(P))\otimes C(X,\Ko(P^{op})))\ar[dr]_{\Psi} \ar[d]_{\widetilde{\ch_P}\times\ch_{P^{op}} } \ar[r]^{\qquad\qquad\Delta^*} & K_*(\Sigma(X,\Ko))\ar[d]^{\ind_X} \\
& H^*_{c,\pi^*\Omega}(T^*X)\times H^*_{-\Omega}(X)   \ar[dr]_{\int_{T^*X}-\wedge Td(X)} & \Z \ar[d]_{} \\
& & \C  &} 
\normalsize
\]
\end{cor}

\subsection{Index formulas on $T$-duals} 

Let us return to the examples of operators coming from $T$-duality. We will in this section calculate the index pairing on the $T$-dual $P^T\to X\times S^1$ of a circle bundle $Z\to X$ in terms of pseudo-differential operators on $Z\times S^1$ using the mapping $\mathfrak{q}_T$ of Proposition \ref{qtconstruct}. 

The main tool in this calculation is the Thom-Connes isomorphism for $\R-C^*$-algebras $A$, this is a natural odd $KK$-isomorphism $\iota_A:A\to A\ltimes \R$. Originally, in \cite{connesthom}, the Thom-Connes isomorphism was constructed as an isomorphism on $K$-theory but it follows from section $10.2.2$ of \cite{cumero} that the Thom-Connes isomorphism comes from a $KK$-isomorphism. We observe that due to the naturality of the Thom-Connes isomorphism, if $A$ and $\hat{A}$ are $\R-C^*$-algebras that are $\R$-equivariantly Poincar\'e dual to each other then $A\ltimes \R$ and $\hat{A}\ltimes \R$ are Poincar\'e dual to each other as well. This in particular implies that for $x\in K_*(A)$ and $y\in K_*(\hat{A})$ then $x.y=\iota_{\hat{A}}(x).\iota_A(y)$, were the dot denotes the index pairing $K_*(A)\times K_*(\hat{A})\to \Z$ defined from the isomorphism $K_*(\hat{A})\cong K^*(A)$ and similarly for $A\ltimes \R$.  

We will use the notation $Z^{op}\to X$ for the circle bundle $Z$ but with reversed orientation on the fiber. The case of interest to us is $A=C(Z^{op})$ and $\hat{A}=C_0(T^*Z)$ so $A\ltimes \R\cong C(X\times S^1,\Ko(P^{T,op}))$ and $\hat{A}\ltimes \R\cong C_0(T^*(X\times S^1),\Ko(\pi^*P^T))$. From the above reasonings we may conclude the following index formula:

\begin{prop}
\label{twistedtdualindex}
Suppose $a\in \Sigma^{0,p}_\mathfrak{U}(X\times S^1,\Ko(P^T))$ is elliptic and $p_\Eg\in C^\infty_\mathfrak{U}(X\times S^1,\Ko(P^{T,op}))$ is a projection, then  
\[\ind \tau_\Eg(a)=\int _{T^*Z} \ch\left(\iota_{T^*Z}^{-1}\Xi[a]\right )\wedge \ch\left(\iota_{Z^{op}}^{-1}[p_\Eg]\right)\wedge Td(X).\]
\end{prop}

An interesting remark here is that since $\Xi[a]$ is of degree $0$ and the Thom-Connes mapping $\iota$ is odd, the form $\ch\left(\iota_{C_0(T^*Z)}^{-1}\Xi[a]\right )$ is an odd degree form on $T^*Z$. 

When the projective symbol is the $T$-dual of a symbol on $Z\times S^1$ we can calculate the index pairing with $K^*(X\times S^1,-\delta(P^T))\cong K^{*-1}(Z)$ explicitly. The linear mapping $\mathfrak{q}_T$ constructed in subsection \ref{tdualityonsymbols} is not a multiplicative mapping, not even up to lower order terms. The degree shift in the index formula of Proposition \ref{twistedtdualindex} comes from this fact. If $a\in S^0(Z\times S^1)$ is an elliptic symbol with parametrix $r$, then because of Proposition \ref{lambdasymbols} we have that 
\[\mathfrak{q}_T(a)\mathfrak{q}_T(r)-1,\mathfrak{q}_T(r)\mathfrak{q}_T(a)-1\in \Sigma^{0,1+}_\mathfrak{U}(X\times S^1,\Ko(P^T))\cap C^{N,\infty}_\mathfrak{U}(T^*(X\times S^1),\ellL^1(P^T)).\]
We define the projection $\kappa_T(a)\in  \Sigma^{0,1+}_\mathfrak{U}(X\times S^1,\Ko(P^T))$ by the formula
\[\kappa_T(a):=\begin{pmatrix} 
-\left((1-\mathfrak{q}_T(a)\mathfrak{q}_T(r)\right)^2+1& \mathfrak{q}_T(a)\left((1-\mathfrak{q}_T(r)\mathfrak{q}_T(a)\right)^2\\
\left((1-\mathfrak{q}_T(r)\mathfrak{q}_T(a)\right)\mathfrak{q}_T(r)& \left((1-\mathfrak{q}_T(r)\mathfrak{q}_T(a)\right)^2
\end{pmatrix}.\]
If $p\in S^0(Z\times S^1)$ is an almost projection, that is $p^2-p\in S^{-1}(Z\times S^1)$, then $\mathfrak{q}_T(p)$ satisfies that 
\[\mathfrak{q}_T(p)^2-\mathfrak{q}_T(p)\in \Sigma^{0,1+}_\mathfrak{U}(X\times S^1,\Ko(P^T))\cap C^{N,\infty}_\mathfrak{U}(T^*(X\times S^1),\ellL^1(P^T)).\]
We define the elliptic projective symbol $\kappa_T(p)\in  \Sigma^{0,1+}_\mathfrak{U}(X\times S^1,\Ko(P^T))$ by 
\[\kappa_T(p):=\exp(2\pi i \mathfrak{q}_T(p)).\]
Using the difference construction, that associates a class in $K^*(T^*(Z\times S^1))$ with a symbol, $\kappa_T$ induces a mapping 
\[\tilde{\kappa}_T:K^*(T^*(Z\times S^1))\to K_*(\Sigma(X\times S^1,\Ko(P^T))).\]
The mapping $\tilde{\kappa}_T$ is defined as mapping the difference class constructed from an elliptic symbol $a$ to $[\kappa_T(a)]-[1]\in K_0(\Sigma(X\times S^1,\Ko(P^T)))$ and similarly the difference class of an almost projection $p\in S^0(Z\times S^1)$ is mapped to $[\kappa_T(p)]\in K_1(\Sigma(X\times S^1,\Ko(P^T)))$. 

Before we express $\tilde{\kappa}_T$ in terms of the Thom-Connes isomorphism we introduce the notation
\[\Delta:T^*(X\times S^1\times S^1)\to T^*(X\times S^1)\] 
for projection onto the first coordinates, so 
\[\Delta_!:K^*(T^*(X\times S^1\times S^1),\pi^*\delta(P^T))\to K^*(T^*(X\times S^1),\pi^*\delta(P^T))\] 
is integration along the fiber. Observe that if we choose a generator for the circle action on $Z$, the Thom-Connes isomorphism $\iota_{T^*(Z\times S^1)}$ can be viewed as a mapping 
\begin{align*}
K^*&(T^*(Z\times S^1))=K^*(\pi^*T^*X\times \R\times S^1\times \R)\to \\
&\to K^{*-1}(T^*X\times S^1\times \R\times S^1\times \R,\pi^*\delta(P^T))=K^{*-1}(T^*(X\times S^1\times S^1),\pi^*\delta(P^T)).
\end{align*}

\begin{sats}
\label{tdualsymbols}
The following diagram commutes:
\[\begin{CD}
K^*(T^*(Z\times S^1))@>\tilde{\kappa}_T>> K_*(\Sigma(X\times S^1,\Ko(P^T)))\\
@V\iota_{T^*(Z\times S^1)} VV @VV \Xi V\\
K^{*-1}(T^*(X\times S^1\times S^1),\pi^*\delta(P^T))@>\Delta_! >> K^{*-1}(T^*(X\times S^1),\pi^*\delta(P^T))\\
\end{CD}\]
In particular, if $p\in S^0(Z\times S^1)$ is an almost projection and $[E]\in K^1(Z)$ is given by $[E]=\iota_{Z^{op}}^{-1}[\Eg^T]$ for a finitely generated projective module $\Eg^T$, then 
\[\ind(\tau_{\Eg^T}(\kappa_T(p)))=\int_{T^*(Z\times S^1)} \ch[p]\wedge \Delta^*\ch[E]\wedge Td(X).\]
\end{sats}

\begin{proof}
The diagram commutes by the construction of $\tilde{\kappa}_T$ since the odd mapping $\iota_{T^*(Z\times S^1)}$ is constructed as the boundary mapping associated with the same extension as $\tilde{\kappa}_T$ is, see section $10.2.2$ of \cite{cumero}. That the diagram commutes merely expresses the fact that the difference construction is natural. 

To prove the index formula we use Proposition \ref{twistedtdualindex} and that $\Xi \circ \tilde{\kappa}_T=\Delta_!\circ \iota_{T^*(Z\times S^1)}$. If $p$ is an almost projection symbol then
\begin{align*}
\ind(\tau_{\Eg^T}(\kappa_T(p)))&=\int _{T^*Z} \ch\left(\iota_{T^*Z}^{-1}\Xi[\kappa_T(p)]\right )\wedge \ch[E]\wedge Td(X)=\\
&= \int _{T^*Z} \ch\left(\iota_{T^*Z}^{-1}\Delta_!\iota_{T^*(Z\times S^1)} [p]\right )\wedge \ch[E]\wedge Td(X)=\\
&=  \int _{T^*(X\times S^1)} \ch_{\pi^*P^T}\left(\Delta_!\iota_{T^*(Z\times S^1)} [p]\right )\wedge \ch_{P^{T,op}}[\Eg^T]\wedge Td(X)=\\
&=\int _{T^*(Z\times S^1)} \ch[a]\wedge \ch\left(\iota_{Z^{op}\times S^1}^{-1}\Delta^*[\Eg^T]\right)\wedge Td(X)=\\
&=\int _{T^*(Z\times S^1)} \ch[a]\wedge \ch\left(\Delta^*[E]\right)\wedge Td(X).
\end{align*}
\end{proof}

If $a$ is an elliptic symbol on $Z\times S^1$ and $[E]\in K^0(Z)$ is of the form $[E]=\iota_{Z^{op}}^{-1}[u]$ for a unitary $u\in C^\infty(X\times S^1,\ellL^1(P^T))$, we obtain by the same method of proof that 
\[\ind_{X\times S^1}(\Xi[\kappa_T(a)]\cup [u])=\int_{T^*(Z\times S^1)} \ch[a]\wedge \Delta^*\ch[E]\wedge Td(X).\]

\subsection{Index pairing for the calculus of subsection \ref{twistedgroupbundle}}

As an example of a twisted index pairing that does not come from the numerical index of a pseudo-differential operator we will in this subsection consider the index pairing for the bundle $P_{Z,\eta}\to X/\Gamma$ as constructed in subsection \ref{twistedgroupbundle} when one pairs elliptic projective symbols coming from projectively equivariant differential operators on $Z$ with a special type of element of $K^*(X/\Gamma,-\mu[\varsigma])$ coming from fiberwise elliptic projectivly equivariant operators on an opposite fiber bundle of $Z$.  For these elements we can express the twisted index pairing in terms of the invariant part of an equivariant index using the Baum-Connes assembly mapping. 

Assume that $Z'\to X$ is another fiber bundle satisfying the same conditions as $Z\to X$ in subsection \ref{twistedgroupbundle}, with the same $M$, and that we have choosen a system $\mathfrak{U}'$ satisfying the same conditions and a vertical $1$-form $\eta$. We denote the associated $C^\infty(X,U(1))$-valued $2$-cocycle on $\Gamma$ by $\varsigma'$. We will in this subsection assume that $[\varsigma]+[\varsigma']=0$ in $H^2(\Gamma, C^\infty(X,U(1)))$ and that $D$ is a fiberwise elliptic operator on $Z'\to X$ that commutes with the projective $\Gamma$-action. An example of a fiberwise operator commuting with the projective action on $Z'$ is $D_{\eta'}:=D_M-i\eta'$.  In this setup, taking fiberwise index defines an element $[D]\in K^0(X/\Gamma,-\mu[\varsigma'])\cong K^0(X/\Gamma,-\mu[\varsigma])$ and we will study the index pairing of $[D]$ with the projective symbols constructed in subsection \ref{projequidiff}.

The condition $[\varsigma]+[\varsigma']=0$ in $H^2(\Gamma, C^\infty(X,U(1)))$ implies that $\varsigma\varsigma'$ is the coboundary of a cochain $\nu\in C^1(\Gamma,C^\infty(X,U(1)))$. That is;
\[\varsigma(\gamma,\gamma')\varsigma'(\gamma,\gamma')=\nu(\gamma) \gamma^*\nu(\gamma')\overline{\nu(\gamma\gamma')}.\] 
We let $\pi_1:Z\times_X Z'\to Z$ and $\pi_2:Z\times_X Z'\to Z'$ denote the projections. Expressed in the functions $(\phi_\gamma)$ and $(\phi_\gamma')$, where the latter is defined on $Z'$ from $\eta'$, this condition takes the form
\[\nu(\gamma\gamma')\e^{i(\pi_1^*\phi_\gamma+\pi_2^*\phi'_\gamma)}\e^{i(\pi_1^*\gamma^*\phi_{\gamma'}+\pi_2^*\gamma^*\phi'_{\gamma'})}=\nu(\gamma)\gamma^*\nu(\gamma')\e^{i(\pi_1^*\phi_{\gamma\gamma'}+\pi_2^*\phi'_{\gamma\gamma'})}.\]
This fact implies that we can define a $\Gamma$-action on $C_b(X,L^2(Z\times_X Z'|X))$ by 
\[\tilde{T}_\gamma:=\nu_\gamma^{-1}\e^{i(\pi_1^*\phi_\gamma+\pi_2^*\phi'_\gamma)}\gamma.\]
Here we let an element $\gamma$ act simply by pullback on $C_b(X,L^2(Z\times_X Z'|X))\subseteq L^2_{loc}(Z\times_X Z')$. In particular, in the case that $Z=Z'$, $\mathfrak{U}=\mathfrak{U}'$ and $\eta=-\eta'$ then $\varsigma\varsigma'=1$ so we can take $\nu=1$.

If $a\in \Sigma^{0,p}_\mathfrak{U}(X/\Gamma,\Ko(P_{Z,\eta}))$ is elliptic with parametrix $r$, we can lift both $a$ and $r$ to $\Gamma$-invariant elements $\tilde{a},\tilde{r}\in \Sigma^{0,p}_\mathfrak{U}(X,\Ko(L^2(Z|X)))$. Consider the operator-valued function $a\#D\in C^\infty(T^*X,\Ko(L^2(Z\times_X Z'|X)))$ defined as
\[a\# D:=\begin{pmatrix} \tilde{a} & D^*/\sqrt{1+D^*D}\\ -D/\sqrt{1+D^*D}& \tilde{r}\end{pmatrix}.\]
The symbol $a\# D$ is clearly $\Gamma$-invariant under the action $\tilde{T}$. Furthermore, $a\#D$ satisfies condition \eqref{firstest} in the local coordinates coming from $\mathfrak{U}$ and $\mathfrak{U}'$. Since $a\#D$ is $\Gamma$-invariant and satisfies a H\"ormander condition we can define the operator $\tau_D(a)\in \Bo(L^2((Z\times_XZ)/\Gamma))$ by the formula
\[\tau_D(a)=(a\#D)(x,\partial).\]
A straight-forward analysis of the indices, the same as in Chapter $7$ of \cite{nasast}, gives the following Proposition.

\begin{prop}
\label{taudlem}
If $a\in \Sigma^p(X/\Gamma,\Ko(P_{Z,\eta}))$ is elliptic, then $\tau_D(a)$ is Fredholm and 
\[\ind \tau_D(a)=\ind_{X/\Gamma}(\Xi[a]\cup [D]).\]
\end{prop}

One of the consequences of this Proposition is an index formula for certain equivariant operators on $Z\times_X Z'$. Suppose that $\mathcal{D}$ is a differential operator on $Z$ equivariant under the projective $\Gamma$-action that is of the same order $k$ as $D$. With $\mathcal{D}$ we can associate its symbol in the horizontal direction $\sigma_{X}(\mathcal{D})$ which is a fiberwise differential operator on $T^*X\times_X Z\to T^*X$, observe that this symbol depends on more than the principal symbol of $\mathcal{D}$. If this differential operator is pointwise invertible outside a $\Gamma$-compact subset of $T^*X$ we say that $\mathcal{D}$ is operator-elliptic over $X$. In the case that $\mathcal{D}$ is operator-elliptic, clearly $a_\mathcal{D}$ is an elliptic projective symbol.

\begin{lem}
There is an operator-elliptic differential operator $\mathcal{D}\#_\Gamma D$ on $(Z\times _X Z')/\Gamma$ of order $k$ such that when extending to a $\tilde{T}$-invariant operator $\mathcal{D}\# D$ on $Z\times _X Z'$ one has that 
\[\sigma_{X}(\mathcal{D}\# D)=\begin{pmatrix} \sigma_{X}(\mathcal{D}) & D^*\\ -D& \sigma_{X}(\mathcal{D}^*)\end{pmatrix}.\]
\end{lem}

The construction of the operator $\mathcal{D}\#_\Gamma D$ comes directly from that the expression for $\sigma_{X}(\mathcal{D}\# D)$ defines a $\Gamma$-invariant differential operator-valued symbol, so it descends to a differential operator-valued symbol on $T^*X/\Gamma$. From this symbol we can construct a differential operator on $(Z\times_X Z')/\Gamma$ which is elliptic since $\mathcal{D}$ is operator-elliptic. The next Theorem follows from Proposition \ref{taudlem}. 

\begin{sats}
\label{projectivetwiseddiscrete}
Suppose that $\mathcal{D}$ is a projectively $\Gamma$-equivariant differential operator on $Z$ that is operator-elliptic over $X$ and that $D$ is a projectively $\Gamma$-equivariant fiberwise elliptic differential operator of the same order on $Z'$. If we let $a_\mathcal{D}\in \Sigma^p_\mathfrak{U}(X/\Gamma,\Ko(P_{Z,\eta}))$ denote the order $0$ elliptic projective symbol associated with $\mathcal{D}$, then 
\[\ind(\mathcal{D}\#_\Gamma D)=\int_{T^*(X/\Gamma)} \widetilde{\ch}_{P_{Z,\eta}}[a_\mathcal{D}]\wedge \ch_{P_{Z,\eta}}[D]\wedge Td(X/\Gamma).\]
\end{sats}


\newpage

\begin{thebibliography}{99}

\bibitem{tommyboy} C.A. Akemann, G.K.  Pedersen, J. Tomiyama, \emph{Multipliers of $C^*$-algebras}, J. Functional Analysis 13 (1973), 277--301. 

\bibitem{atiyahsegaltwist} M.F. Atiyah, G. Segal, \emph{Twisted K-theory}, Ukr. Mat. Visn. 1 (2004), no. 3, 287--330; translation in Ukr. Math. Bull. 1 (2004), no. 3, 291--334

\bibitem{atiyahsegalcohom} M.F. Atiyah, G. Segal, \emph{Twisted K-theory and cohomology}, Inspired by S. S. Chern, 5--43, Nankai Tracts Math., 11, World Sci. Publ., Hackensack, NJ, 2006. 

\bibitem{atiyahsingerI} M.F. Atiyah, I.M. Singer, \emph{The index of elliptic operators I},  Ann. of Math. (2)  87, 1968, p. 484--530. 

\bibitem{atiyahsingerIII} M.F. Atiyah, I.M. Singer, \emph{The index of elliptic operators. III},  Ann. of Math. (2)  87, 1968, p. 546--604.

\bibitem{baumdtva} P. Baum, R.G. Douglas, \emph{$K$ homology and index theory}, Operator algebras and applications, Part I (Kingston, Ont., 1980),  p. 117--173, Proc. Sympos. Pure Math., 38, Amer. Math. Soc., Providence, R.I., 1982.

\bibitem{baumvanerp} P.F. Baum, E. van Erp, \emph{K-homology and index theory on contact manifolds}, arXiv:1107.1741v1.

\bibitem{blackadaropalg} B. Blackadar,  \emph{Operator algebras. Theory of $C\sp *$-algebras and von Neumann algebras}, Encyclopaedia of Mathematical Sciences, 122. Operator Algebras and Non-commutative Geometry, III. Springer-Verlag, Berlin, 2006.

\bibitem{bocamava} P. Bouwknegt, A.L. Carey, V. Mathai, M.K. Murray, D. Stevenson, \emph{Twisted $K$-theory and $K$-theory of bundle gerbes},  Comm. Math. Phys.  228  (2002),  no. 1, p. 17--45.

\bibitem{bmrs} J. Brodzki, V. Mathai, J. Rosenberg, R.J. Szabo, \emph{$D$-Branes, $RR$-Fields and Duality on Noncommutative Manifolds}, arXiv:hep-th/0607020v3.

\bibitem{carwang} A. L. Carey, B-L Wang, \emph{Riemann-Roch and index formulae in twisted K-theory}, arXiv:0909.4848v1.

\bibitem{connesthom} A. Connes, \emph{An analogue of the Thom isomorphism for crossed products of a $C^*$-algebra by an action of $\R$}, Adv. in Math. 39 (1981), no. 1, 31--55. 

\bibitem{cumero} J. Cuntz, R. Meyer, J.M. Rosenberg, \emph{Topological and Bivariant $K$-theory}, Birkh�user 2007.

\bibitem{dixmierdouady} J. Dixmier, A. Douady, \emph{Champs continus d'espaces hilbertiens et de $C^{\ast} $-algebres},  Bull. Soc. Math. France  91 (1963), p. 227--284. 

\bibitem{schulze} C. Dorschfeldt, B.W. Schulze, \emph{Pseudo-differential operators with operator-valued symbols in the Mellin-edge-approach}, Ann. Global Anal. Geom. 12 (1994), no. 2, p. 135--171.

\bibitem{epmel} C.L. Epstein, R.B. Melrose, \emph{Shrinking tubes and the $\bar{\partial}$-Neumann problem}. Unpublished. Available at http://www.math.upenn.edu/$\sim$cle/papers/index.html, 1990. 

\bibitem{fhtii} D.S. Freed, M.J. Hopkins, C. Teleman, \emph{Loop Groups and Twisted $K$-Theory II}, arXiv:math/0511232v2.

\bibitem{mgoffpv} M. Goffeng, \emph{The Pimsner-Voiculescu sequence for coactions of compact Lie groups}, arXiv:1004.4333v4, to appear in Mathematica Scandinavica.

\bibitem{gorokh} A. Gorokhovsky, \emph{Characters of cycles, equivariant characteristic classes and Fredholm modules},  Comm. Math. Phys.  208  (1999),  no. 1, p. 1--23. 

\bibitem{guille} V. Guillemin, \emph{Toeplitz operators in $n$ dimensions}, Integral Equations Operator Theory 7 (1984), no. 2, p. 145--205.

\bibitem{hormtre} L. H\"ormander, \emph{The analysis of linear partial differential operators. III. Pseudo-differential operators}, Reprint of the 1994 edition. Classics in Mathematics. Springer, Berlin, 2007. 

\bibitem{kaspindex} G.G. Kasparov, \emph{Operator K-theory and its applications: elliptic operators, group representations, higher signatures, $C^*$-extensions}, Proceedings of the International Congress of Mathematicians, Vol. 1, 2 (Warsaw, 1983), 987--1000, PWN, Warsaw, 1984.

\bibitem{kuiper} N.Kuiper, \emph{Contractibility of the Unitary Group in Hilbert Space}, Topology 3 (1964), p. 19--30. 

\bibitem{luke} G. Luke, \emph{Pseudodifferential operators on Hilbert bundles}, J. Differential Equations 12 (1972), p. 566--589. 

\bibitem{mmett} M. Marcolli, V. Mathai, \emph{Twisted index theory on good orbifolds. I. Noncommutative Bloch theory},  Commun. Contemp. Math.  1  (1999),  no. 4, p. 553--587.


\bibitem{mathstev} V. Mathai, D. Stevenson, \emph{On a generalized Connes-Hochschild-Kostant-Rosenberg theorem}, Adv. Math. 200 (2006), no. 2, p. 303--335. 


\bibitem{murray} M. K. Murray, \emph{Bundle gerbes}, J. London Math. Soc. (2), 54(2), p. 403--416, 1996.

\bibitem{nasast} V.E. Nazaikinskii, A.Y. Savin, B.Y. Sternin, \emph{Elliptic theory and noncommutative geometry}, Nonlocal elliptic operators. Operator Theory: Advances and Applications, 183. Advances in Partial Differential Equations (Basel). Birkh\"auser Verlag, Basel, 2008.

\bibitem{raeros} I. Raeburn, J. Rosenberg, \emph{Crossed products of continuous trace $C^*$-algebras by smooth actions}, Trans. Amer. Math. Soc. 305 (1988), no. 1, p. 1--45.

\bibitem{rosentexas} J. Rosenberg, \emph{Topology, $C^\ast$-algebras, and string duality}, CBMS Regional Conference Series in Mathematics, 111. Published for the Conference Board of the Mathematical Sciences, Washington, DC; by the American Mathematical Society, Providence, RI, 2009.

\bibitem{grigone} G. Rozenblum, \emph{On some analytical index formulas related to operator-valued symbols}, Electron. J. Differential Equations 2002, No. 17. 

\bibitem{grigtwo} G. Rozenblum, \emph{Regularisation of secondary characteristic classes and unusual index formulas for operator-valued symbols}, Nonlinear hyperbolic equations, spectral theory, and wavelet transformations, p. 419--437, Oper. Theory Adv. Appl., 145, Birkh\"auser, Basel, 2003. 

\bibitem{tutwisted} J.-L. Tu, \emph{Twisted K-theory and Poincar\'e duality}, Trans. Amer. Math. Soc. 361 (2009), no. 3, 1269--1278. 




\end{thebibliography}
\end{document}